\documentclass[10pt,a4paper,tbtags,leqno]{amsart} 
\usepackage[top=20mm,bottom=20mm,left=20mm,right=20mm]{geometry} 
\usepackage[dvipdfmx]{graphicx} 
\usepackage{amsmath,amssymb,amsthm} 
\usepackage{mathtools} 
\mathtoolsset{showonlyrefs=true} 
\usepackage{bm} 
\usepackage{comment} 
\usepackage{float} 
\usepackage{ascmac} 
\usepackage{color} 
\usepackage{cite} 
\usepackage{enumerate} 
\usepackage{multirow} 
\usepackage{mathrsfs} 
\usepackage{empheq} 
\usepackage{pgfplotstable} 
\usepackage{pgfplots} 
\pgfplotsset{compat=1.18} 
\usepackage{booktabs} 
\newtheorem{theorem}{Theorem}[section] 
\newtheorem{lemma}[theorem]{Lemma}

\theoremstyle{definition} 
 
\numberwithin{equation}{section} 
\usepackage{apptools} 
\usepackage[colorlinks=false, hidelinks=true]{hyperref} 
\title[Error estimates of fully semi-Lagrangian schemes]{Error estimates of fully semi-Lagrangian schemes for diffusive conservation laws} 
\author{Haruki Takemura} 
\email{takemura@ms.u-tokyo.ac.jp}
\thanks{This work was supported by JSPS KAKENHI Grant Number JP24KJ0964. } 
\date{} 
\def\norm#1#2{\left\|#2\right\|_{#1}} 
\def\snorm#1#2{\left|#2\right|_{#1}} 
\newcommand{\relmiddle}{\mathrel{}\middle|\mathrel{}} 
\begin{document} 

\begin{abstract} 
  We present error estimates of the fully semi-Lagrangian scheme with high-order interpolation operators, solving the initial value problems for the one-dimensional nonlinear diffusive conservation laws, including the Burgers equations. We impose certain assumptions on the interpolation operator, which are satisfied by both spline and Hermite interpolations. We establish the convergence rates of $ O(\Delta t + h^{2 s} / \Delta t) $ in the $ L^2 $-norm and $ O(\Delta t + h^{s} / (\Delta t)^{1/2} + h^{2s} / \Delta t) $ in the $ H^s $-norm for the spatial mesh size $ h $ and the temporal step size $ \Delta t $, where the spline or Hermite interpolation operator of degree $ (2s - 1) $ is employed. The numerical results are in agreement with the theoretical analysis. 
\end{abstract} 

\maketitle 

\section{Introduction} 

Semi-Lagrangian methods are widely used for the numerical approximation of conservation laws. One of the main advantages of the semi-Lagrangian framework is that it is free from the Courant--Friedrichs--Lewy condition, which allows for relatively large time steps and improved computational efficiency. As a result, semi-Lagrangian methods have been applied to a broad range of problems including atmospheric simulations \cite{91StCo} and plasma physics \cite{99SoRoBeGh}. 

In this paper, we consider the initial value problem of the nonlinear advection--diffusion equation in the periodic domain $ \mathbb{T} = \mathbb{R} / \mathbb{Z} $: 
\begin{align} \label{ad}
  \begin{cases} 
    \partial_t u + f(u) \, \partial_x u - \nu \, \partial_x^2 u = 0 & \text{in} \quad \mathbb{T} \times [0, T], \\
    u (0) = u_0 & \text{in} \quad \mathbb{T}, 
  \end{cases} 
\end{align} 
where $ u_0 $ is the given initial data, and $ \nu $ is a positive constant. The given flux $ f $ satisfies $ c_{f} \leq \snorm{}{f^\prime} \leq C_{f} $ for some positive constants $ c_{f}, C_{f} $. In this paper, we assume that $ f $ and $ u_0 $ are sufficiently smooth so that the problem \eqref{ad} admits a unique, sufficiently smooth solution $ u $. 

We now introduce the numerical scheme studied in this paper. Let $ N_t, N_x \in \mathbb{N} $, $ \Delta t = T / N_t $, $ t_n = n \Delta t $. Let $ 0 = x_0 < x_1 < \cdots < x_{N_x} = 1 $ denote a spatial grid on $ \mathbb{T} $. We define the mesh size as $ h = \max_{m} |x_{m} - x_{m-1}| $. The fully semi-Lagrangian scheme for \eqref{ad} proposed in \cite{25Fe} is given by 
\begin{gather} 
  \label{dSL-V} 
  \begin{cases} 
    u_m^0 = u_0 (x_m), \\ 
    \begin{aligned}
    u_m^{n} & = \frac{1}{2} \tilde{\mathcal{I}}_h [U^{n - 1}] \left( x_m - f (u_m^n) \Delta t - \sqrt{2} \delta \right) \\ 
    & \quad + \frac{1}{2} \tilde{\mathcal{I}}_h [U^{n - 1}] \left( x_m - f (u_m^n) \Delta t + \sqrt{2} \delta \right),   
    \end{aligned}
  \end{cases} 
\end{gather} 
where $ \delta = (\nu \, \Delta t)^{1/2} $, $ u_m^n $ denotes the approximation of $ u (x_m, t_n) $, $ U^n = (u_1^n, \ldots, u_{N_x}^n) $, and $ \tilde{\mathcal{I}}_{h} [U^n] $ is an interpolation of $ U^n $ that satisfies $ \tilde{\mathcal{I}}_{h} [U^n] (x_m) = u_m^n $ for all $ m \in \{1, \ldots, N_x\} $. The simplest choice for the interpolation operator $\mathcal{I}_h$ is the linear interpolation. To improve the accuracy, semi-Lagrangian schemes employing higher-order interpolation operators have been explored for several decades. For example, the use of cubic spline interpolation was proposed in \cite{76Pu}. 

We recall some notable features of the scheme~\eqref{dSL-V}. 

First, following the semi-Lagrangian approach, the solution at each grid point $ x_m $ is computed by tracing backward along the characteristic curve over the time interval $ [t_{n-1}, t_n] $. 

Second, the scheme is implicit in time: the velocity $ f(u) $ used for backtracking is evaluated at the unknown value $ u_m^n $. Although this implicit formulation increases the computational cost, it significantly improves the robustness of the method. Moreover, each nonlinear equation in the second line of \eqref{dSL-V} involves only a single scalar unknown $ u_m^n $. Thanks to this locality, the scheme can be efficiently implemented using a simple Newton iteration, and it is also appropriate for parallel computation. 

Third, the diffusive term is also treated in a semi-Lagrangian technique, and that is why this scheme is referred to as a fully semi-Lagrangian method. The derivation of this scheme is based on the Feynman--Kac formula. The idea of constructing deterministic numerical schemes for parabolic problems via stochastic representations is discussed in detail by~\cite{02Mi}. 

A mathematical analysis of the scheme was conducted in~\cite{25Fe}, established the following results: (i) an estimate of the consistency error, (ii) $ L^\infty $ -stability under the assumption that the interpolation operator does not increase the $ L^\infty $-norm, and (iii) one-sided Lipschitz stability under the condition that the interpolation operator does not increase the $ W^{1,\infty} $-seminorm. 

Schemes closely related to \eqref{dSL-V} have been already proposed in \cite{00MiTr,01MiTr,02Mi}. They share the same fundamental structure with \eqref{dSL-V}, but the implicit step is replaced by an explicit one. For these explicit versions, error estimates for the fully discrete scheme are established in the case of linear interpolation. See also \cite{00MiTr} for a similar implicit scheme applied to a broader class of PDEs with variable viscosity coefficients, as well as its $ L^\infty $-error estimate for a semi-discrete version. 

However, as noted in Remark 5.2 of \cite{02Mi}, a direct extension of their strategy to the case of higher-order spline interpolation is not straightforward. This difficulty is one of the main motivations for the present paper. In our study, we address the semi-Lagrangian methods with higher-order interpolation, including the spline interpolation of degree three and above. A crucial step in the analysis is to establish the stability of the interpolation operator with respect to both the standard $ H^s $-norm and the weighted $ H^s $-norm (defined in Subsection \ref{ss:pre}) for some $ s $. 

When the cubic Hermite interpolation operator is used in the semi-Lagrangian scheme, the resulting method is known as the cubic interpolated pseudo-particle (CIP) scheme, which is originally proposed in \cite{85TaNiYa}. Applications of the CIP scheme can be found in numerous works, including its use for the Korteweg--de Vries equations \cite{91YaAo}, and the shallow water equations \cite{09ToOgYa}. 

However, mathematical analyses of CIP schemes remain limited. The stability of the CIP scheme for the linear conservation law with a constant velocity was investigated in \cite{02NaTa, 15TaFuNiIs}. As stated in \cite{15TaFuNiIs}, it seems to be difficult to establish the $ L^2 $-stability of the CIP scheme. In our previous study \cite{24KaTa}, we derived the $ L^2 $-error estimate of order $ O(\Delta t^3 + h^4/\Delta t) $. In the present paper, we demonstrate that the higher-order Hermite interpolation operator can be employed as $ \mathcal{I}_h $, and that our error estimate also holds in this case. 

This paper is organized as follows. In Section~\ref{s:2}, we reformulate the scheme \eqref{dSL-V}, originally described in terms of grid functions, using a finite-dimensional subspace of $ H^s (\mathbb{T}) $. We also introduce some notation in this section. In Section~\ref{s:3}, we present the $ H^s $- and $ L^2 $-error analyses, which are the main results of this paper. In Section~\ref{s:4}, we present some numerical results. In the appendix, we provide auxiliary lemmas used in the proofs in Section~\ref{s:3}. 

\section{The fully semi-Lagrangian scheme for diffusive conservation laws} \label{s:2} 

\subsection{Reformulation of the semi-Lagrangian scheme} 

We recast the scheme \eqref{dSL-V}, originally formulated in terms of grid functions, as a scheme based on functions on $ \mathbb{T} $. For $ v: \mathbb{T} \to \mathbb{R} $, the approximate characteristics $ {X_{\Delta t}^1}[v] : \mathbb{T} \to \mathbb{T} $ is defined by 
\begin{align} 
  {X_{\Delta t}^1}[v] (x) = x - v(x) \Delta t. 
\end{align} 
Then we can rewrite \eqref{dSL-V} as 
\begin{align} 
  \label{SL-V} 
  \begin{cases} 
    u_h^0 = \mathcal{I}_h u_0, & \\ 
    u_h^{n} = \mathcal{I}_h \left[\mathcal{S}_{\Delta t}^{\mathrm{V}} (\mathcal{S}_{\Delta t}^{\mathrm{A}} u_h^{n - 1})\right] & (n = 1, \ldots, N_t), 
  \end{cases} 
\end{align} 
where the interpolation operator $ \mathcal{I}_h: C(\mathbb{T}) \to V_h^{2s-2, 2s-1} $ is defined by 
\begin{align} 
  \mathcal{I}_h v = \tilde{\mathcal{I}}_h [\left(v(x_1), \ldots, v(x_{N_x})\right)], 
\end{align} 
$ \mathcal{S}_{\Delta t}^{\mathrm{V}}: H^s (\mathbb{T}) \to H^s (\mathbb{T}) $ is defined by 
\begin{align} 
  \mathcal{S}_{\Delta t}^{\mathrm{V}} (v) & = \frac{1}{2} \left(v \circ {X_{\Delta t}^1}[\tilde{\delta}] + v \circ {X_{\Delta t}^1}[-\tilde{\delta}]\right) 
\end{align} 
for $ \tilde{\delta} = \sqrt{2} \delta / \Delta t $, and $ \mathcal{S}_{\Delta t}^{\mathrm{A}} : H^s (\mathbb{T}) \to H^s (\mathbb{T}) $ is implicitly defined by 
\begin{align} 
  \mathcal{S}_{\Delta t}^{\mathrm{A}} (v) & = v \circ {X_{\Delta t}^1} [f \circ \mathcal{S}_{\Delta t}^{\mathrm{A}} (v)]. \label{def_sa} 
\end{align} 
We can deduce that the operator $ \mathcal{S}_{\Delta t}^{\mathrm{A}} $ is well-defined when $ \Delta t $ is sufficiently small from Lemmas~\ref{lem:s} and \ref{lem:s2}. The schemes \eqref{dSL-V} and \eqref{SL-V} are equivalent in the sense that $ \mathcal{I}_h [U^{n}] = u_h^n $. 

In practice, we often use spline interpolation operator as $ \mathcal{I}_h $. We denote each cell by $ E_m = [x_m, x_{m + 1}] $. For $ k $, $ l \in \mathbb{N} \cup \{0\} $, we set the finite-dimensional function space 
\begin{align} 
  V_h^{k, l} \coloneqq \left\{v \in C^{k} (\mathbb{T}) \relmiddle v|_{E_m} \in \mathbb{P}_{l} (E_m), \quad m  = 1, \ldots, N_x \right\}, 
\end{align} 
where $ \mathbb{P}_{l} (E_m) $ is the set of all polynomials of degree at most $ l $ on $ E_m $. For $ v \in C (\mathbb{T}) $, the spline interpolation of degree $ (2 s - 1) $, denoted by $ \mathcal{I}_h^{\mathrm{S}} v $, is defined by $ \mathcal{I}_h^{\mathrm{S}} v \in V_h^{2s-2, 2s-1} $ and $ \mathcal{I}_h^{\mathrm{S}} v (x_m) = v (x_m) $. The condition on the spatial mesh ensuring the existence and uniqueness of the spline interpolation is discussed in Theorem 4 of \cite{64AhNiWa}. In this paper, we assume that the interpolation exists uniquely. 

\subsection{Semi-Lagrangian scheme with Hermite interpolation} 

In this subsection, we consider using the Hermite interpolation as the interpolation in the scheme \eqref{SL-V}. In the following, we see that the scheme is implementable with the Hermite interpolation. The $(2s-1)$-th degree Hermite interpolation operator $\mathcal{I}_h^{\mathrm{H}}: C^{s-1}(\mathbb{T}) \to V_h^{s-1, 2s-1}$ is defined by 
\begin{align} 
  \partial_x^k (\mathcal{I}_h^{\mathrm{H}} v) (x_j) = \partial_x^k v(x_j) \quad (k = 0,\ldots,s - 1). 
\end{align} 
The resulting semi-Lagrangian scheme reads 
\begin{align} 
  \begin{cases} 
    u_h^0 = \mathcal{I}_h^{\mathrm{H}} u_0, & \\ u_h^{n} = \mathcal{I}_h^{\mathrm{H}} \left[\mathcal{S}_{\Delta t}^{\mathrm{V}} (\mathcal{S}_{\Delta t}^{\mathrm{A}} u_h^{n - 1})\right] & (n = 1, \ldots, N_t). 
  \end{cases} \label{SL-V-Hermite} 
\end{align} 

To illustrate this concretely, let us consider the case $s = 2$. Let $ v_m^n $ approximate $ \partial_x u (x_m, t_n) $ and $ \tilde{U}^n $ denote the vector $ (u_1^n, \ldots, u_{N_x}^n, v_1^n, \ldots, v_{N_x}^n) $. We denote the cubic Hermite interpolation operator by $ \tilde{\mathcal{I}}_h^{\mathrm{H3}}: \mathbb{R}^{2N_x} \to V_h^{1, 3} $. In analogy to the expression of \eqref{dSL-V}, we can rewrite the corresponding grid function scheme as 
\begin{align} 
  u_m^{n} & = \frac{1}{2} \tilde{\mathcal{I}}_h^{\mathrm{H3}} [\tilde{U}^{n - 1}] \left(z_m^{n-}\right) + \frac{1}{2} \tilde{\mathcal{I}}_h^{\mathrm{H3}} [\tilde{U}^{n - 1}] \left( z_m^{n+} \right), \\ 
  v_m^n & = \frac{1}{2} \left[\partial_x \tilde{\mathcal{I}}_h^{\mathrm{H3}} [\tilde{U}^n] (z_m^{n-}) + \partial_x \tilde{\mathcal{I}}_h^{\mathrm{H3}} [\tilde{U}^n] (z_m^{n+})\right] \\ 
  & \quad \times \left[ 1 + \frac{\Delta t}{2} f^\prime (u_m^n) \left(\partial_x \tilde{\mathcal{I}}_h^{\mathrm{H3}} [\tilde{U}^{n-1}] (z_m^{n-}) + \partial_x \tilde{\mathcal{I}}_h^{\mathrm{H3}} [\tilde{U}^{n-1}] (z_m^{n+})\right) \right]^{-1}, \label{dSL-V-Hermite} 
\end{align} 
where $ z_m^{n\pm} = x_m - f (u_m^n) \Delta t \pm \sqrt{2} \delta $. The second equation in \eqref{dSL-V-Hermite} is obtained by differentiating the second line of \eqref{SL-V-Hermite}. As in the spline interpolation case, the implicit scheme \eqref{SL-V-Hermite} can also be implemented with Hermite interpolations of degree higher than the example above. 

In the remainder of this paper, we assume that $ \mathcal{I}_{h}: H^s(\mathbb{T}) \to H^s(\mathbb{T}) $ is a general interpolation operator that satisfies the following three properties: 
\begin{enumerate}[(P1)] 
  \item The operator $ \mathcal{I}_{h} $ satisfies the integral relation \label{p:o} 
  \begin{align} 
    \snorm{s, 2}{(I - \mathcal{I}_{h}) v - \mathcal{I}_{h} w}^2 & = \snorm{s, 2}{(I - \mathcal{I}_{h}) v}^2 + \snorm{s, 2}{\mathcal{I}_{h} w}^2 
  \end{align} 
  for any $ v, w \in H^s (\mathbb{T}) $, where we denote the identity operator by $ I $. 
  \item There exists a constant $ C_s > 0 $ such that for any $ v \in H^s (\mathbb{T}) $, \label{p:s}
  \begin{align}
    \norm{0, 2}{(I - \mathcal{I}_{h}) v} & \leq C_s h^s \snorm{s, 2}{v}. \label{interpolation_error_s}
  \end{align}
  \item For an integer $ q > s $, there exists a constant $ C_q > 0 $ such that for any $ v \in H^q (\mathbb{T}) $, \label{p:q} 
  \begin{align} 
    \snorm{s, 2}{(I - \mathcal{I}_{h}) v} & \leq C_q h^{q - s} \snorm{q, 2}{v}. 
  \end{align} 
\end{enumerate} 

We remark that the spline interpolation of degree $ (2 s - 1) $ and the Hermite interpolation of degree $ (2 s - 1) $ satisfy (P\ref{p:o}), (P\ref{p:s}) and (P\ref{p:q}) for $q = 2s$. These three facts can be found in Theorems 4, 7 and 9 in~\cite{SV67}, respectively. 

\subsection{Preliminaries} 
\label{ss:pre} 

We introduce some notation. The norms of Sobolev spaces $ L^p (\mathbb{T}) $ and $ W^{s, p} (\mathbb{T}) $ are denoted by $ \norm{0, p}{\, \cdot \,} $ and $ \norm{s, p}{\, \cdot \,} $, respectively. We also use the Sobolev seminorm $ \snorm{s, p}{\, \cdot \,} = \norm{0, p}{\partial_x^s (\, \cdot \,)} $. For a positive integer $ s $, we define 
\begin{align} 
  \norm{s, 2, \ast}{v} = \left(\norm{0,2}{v}^2 + \snorm{s,2}{v}^2\right)^{1/2}, 
\end{align} 
which is a norm equivalent to the standard $ H^s $-norm. We simply refer to $ \norm{s, 2, \ast}{\, \cdot \,} $ as the $ H^s $-norm, when there is no risk of confusion. We further define the norm $ \norm{s, 2, \Delta}{\,\cdot\,} $, which depends on the mesh size $ h $ and the time step size $ \Delta t $, by 
\begin{align} 
  \norm{s, 2, \Delta}{v} = \left(\norm{0, 2}{v}^2 + \frac{h^{2s}}{\Delta t} \snorm{s, 2}{v}^2 \right)^{1/2}, \label{nsd} 
\end{align} 
and refer to it as the weighted $ H^s $-norm. For a Sobolev space $ H^r (\mathbb{T}) $, we write $W^{k,p} (H^r)$ to denote the Bochner space $W^{k,p}(0,T; H^r(\mathbb{T}))$. 

We use a generic positive constant $ C $ that may change from line to line.  In addition, we use $ C_{\mathrm{P}} (A_1, \ldots, A_l) $ to denote a generic positive constant that depends polynomially on $ (A_1, \ldots, A_l) $. That is, there exists a polynomial $ \Phi (A_1, \ldots, A_l) $ such that 
\begin{align} 
  C_{\mathrm{P}} (A_1, \ldots, A_l) \leq \Phi (A_1, \ldots, A_l). 
\end{align} 

The following Sobolev embedding inequality will be used frequently. There exists a positive constant $ C $ such that 
\begin{align}
  \norm{0, \infty}{\phi} \leq C \norm{1, 2}{\phi}. \label{a23-emb} 
\end{align} 

\subsection{Well-definedness of $\mathcal{S}_{\Delta t}^{\mathrm{A}}$} \label{ss:2-3} In the following lemma, we show that the operator $ \mathcal{S}_{\Delta t}^{\mathrm{A}} $ is well-defined by \eqref{def_sa}, provided that $ \Delta t $ is sufficiently small. 

\begin{lemma} \label{lem:s} 
  Let $ v \in H^s (\mathbb{T}) $ and $ f \in W^{s, \infty} (\mathbb{T}) $ for an integer $ s \geq 2 $. We define $ \Delta t_1 = \min\{(3 \snorm{1, \infty}{f} \snorm{1, \infty}{v})^{-1}, 1\} $. If $ \Delta t < \Delta t_1 $, then there exists a unique $ w \in H^s (\mathbb{T}) $ such that 
  \begin{align} 
    w = v \circ {X_{\Delta t}^1}[f \circ w]. \label{a1} 
  \end{align} 
  Furthermore, $ {X_{\Delta t}^1}[f \circ w] $ is a bijection. 
\end{lemma} 

\begin{proof} 
  We define $ G: \mathbb{R} \times \mathbb{T} \to \mathbb{R} $ by 
  \begin{align}
    G (\xi, x) = \xi - v (x - f (\xi) \Delta t). 
  \end{align}
  It follows that 
  \begin{align}
    \partial_\xi G (\xi, x) = 1 - f^\prime (\xi) (\partial_x v) (x - f (\xi) \Delta t) \Delta t, 
  \end{align} 
  Suppose $ \Delta t \in (0, \Delta t_1) $. Then we have $ \partial_\xi G \geq 2 / 3 $. 
  By the implicit function theorem, for each $ x \in \mathbb{T} $, there exists a unique $ \xi(x) \in \mathbb{R} $ such that $ G (\xi(x), x) = 0 $. 
  This implies the unique existence of a function $ w: \mathbb{T} \to \mathbb{R} $ satisfying \eqref{a1}. 

  By differentiating \eqref{a1}, we have 
  \begin{align}
    \gamma \partial_x w = (\partial_x v) \circ {X_{\Delta t}^1} [f \circ w], \label{a10} 
  \end{align}
  where 
  \begin{align}
     \gamma = 1 + \left((f^\prime) \circ w\right) \left((\partial_x v) \circ {X_{\Delta t}^1} [f \circ w]\right) \Delta t. \label{a35} 
  \end{align}
  Since it holds that $ 1 / (1 - a) \leq 1 + (3/2)a $ for $ a \in (0, 1/3) $, when $ \Delta t < \Delta t_1 $, we have 
  \begin{align} 
    \norm{0, \infty}{1 / \gamma} & \leq \frac{1}{1 - \Delta t \snorm{1, \infty}{f} \snorm{1, \infty}{v}} \\
    & \leq 1 + \frac{3}{2} \Delta t \snorm{1, \infty}{f} \snorm{1, \infty}{v}. \label{a11}
  \end{align} 
  Taking the $ L^\infty $-norm of \eqref{a10}, and applying \eqref{a11}, we obtain 
  \begin{align} 
    \snorm{1, \infty}{w} 
    & \leq \norm{0, \infty}{1 / \gamma} \norm{0,\infty}{\gamma \, \partial_x w} \\
    & \leq \left(1 + \frac{3}{2} \snorm{1, \infty}{f} \snorm{1, \infty}{v} \Delta t\right) \snorm{1,\infty}{v}. \label{a5} 
  \end{align} 
  Combining with $ 3 \snorm{1,\infty}{f} \snorm{1,\infty}{v} \Delta t < 1 $, we have 
  \begin{align} 
    \snorm{1, \infty}{w} \leq \frac{3}{2} \snorm{1, \infty}{v}. \label{a5-2}
  \end{align}
  By \eqref{a5-2} and $ \Delta t \leq \Delta t_1 $, we have $ \Delta t \snorm{1, \infty}{f} \snorm{1, \infty}{w} \leq 1/2 $. Combining with 
  \begin{align} 
    \partial_x {X_{\Delta t}^1} [f \circ w] = 1 - (f^\prime \circ w) \partial_x w \Delta t, 
  \end{align} 
  we obtain 
  \begin{align} 
    \frac{1}{2} \leq \partial_x {X_{\Delta t}^1} [f \circ w] (x) \leq \frac{3}{2} 
  \end{align} 
  for any $ x \in \mathbb{T} $, which implies $ {X_{\Delta t}^1} [f \circ w]: \mathbb{T} \to \mathbb{T} $ is a bijection. 
\end{proof} 

\section{Main results} 
\label{s:3} 

\subsection{Consistency} 

In this section, we present the $ L^2 $-error estimate for the scheme \eqref{SL-V}. To this end, we first establish the $H^s$-error estimate. We begin by investigating the consistency of the scheme in the $ L^2 $-norm and the $ H^s $-seminorm. 

We denote the exact solution $ u (t_n) $ by $ u^n $. We define the truncation error function $ \tau_h^n: \mathbb{T} \to \mathbb{R} $ ($ n = 1, \ldots, N_t $) by 
\begin{align} 
  \tau_h^n & = \frac{1}{\Delta t} \left\{u^n - (\mathcal{S}_{\Delta t}^{\mathrm{V}} u^{n - 1}) \circ {X_{\Delta t}^1} [f \circ u^n]\right\} \\ 
  & = \frac{1}{\Delta t} \left\{u^n - \frac{1}{2} ( u^{n - 1} \circ {X_{\Delta t}^1} [f \circ u^n - \tilde{\delta}] + u^{n - 1} \circ {X_{\Delta t}^1} [f \circ u^n + \tilde{\delta}]) \right\}, \label{def_tau} 
\end{align} 
where the interpolation error is deliberately excluded from the definition. We will show that $ \tau_h^n $ is of order $ O(\Delta t) $ under suitable regularity assumptions. 

Suppose that $ u \in W^{2,\infty} (0, T; H^{s} (\mathbb{T})) \cap L^{\infty} (0, T; H^{s + 4} (\mathbb{T})) $. 
We expand $ u^{n - 1} \circ {X_{\Delta t}^1} [u^n \pm \tilde{\delta}] $ by Taylor's theorem and apply the first equality of \eqref{ad} to obtain 
\begin{align} 
  \tau_h^n & = \tau_{h1}^n + \tau_{h2}^n + \tau_{h3}^n + \tau_{h4+}^n + \tau_{h4-}^n, \label{t0}
\end{align} 
where 
\begin{align}
  \tau_{h1}^n & = \frac{1}{\Delta t} \left( u^{n - 1} - u^n \right) + \partial_t u^{n - 1}, \\
  \tau_{h2}^n & = u^{n - 1} \partial_x u^{n - 1} - u^n \partial_x u^{n - 1}, \\
  \tau_{h3}^n & = \Delta t \left\{\frac12 f(u^n)^2 \partial_x^2 u^{n - 1} - \left(\frac13 \Delta t f (u^n)^3 + 6 \nu f (u^n) \right) \partial_x^3 u^{n - 1}\right\} 
\end{align}
and 
\begin{align}
  \tau_{h4\pm}^n (x) & = - \frac{\Delta t}{6} \int_0^{f (u^n (x)) \sqrt{\Delta t} \mp \sqrt{2 \nu}} \partial_x^4 u^{n - 1} (x - \sqrt{\Delta t} z) z^3 \, \mathrm{d}z. 
\end{align}

Each term on the right-hand side of \eqref{t0} will be estimated below. By Taylor's theorem, we have 
\begin{align} 
  \tau_{h1}^n & = \frac{1}{\Delta t} \int_{0}^{\Delta t} t^\prime \partial_t^2 u (t^{n - 1} + t^\prime) \, \mathrm{d} t^\prime. 
\end{align} 
Therefore, we obtain 
\begin{align}
  \norm{0, 2}{\tau_{h1}^n} & \leq C \Delta t \norm{W^{2, \infty}(t_{n - 1}, t_n; L^2 (\mathbb{T}))}{u}, \label{t10}
\end{align}
and 
\begin{align}
  \snorm{s, 2}{\tau_{h1}^n} & \leq C \Delta t \norm{W^{2, \infty}(t_{n - 1}, t_n; H^s (\mathbb{T}))}{u}. \label{t1s}
\end{align}

By the fundamental theorem of calculus, we have 
\begin{align} 
  \tau_{h2}^n & = - \partial_x u^{n - 1} \int_{0}^{\Delta t} \partial_t u (t^{n - 1} + t^\prime) \, \mathrm{d} t^\prime.  
\end{align}
Thus, we obtain 
\begin{align}
  \norm{0, 2}{\tau_{h2}^n} & \leq C \Delta t \norm{1, \infty}{u^{n - 1}} \norm{W^{1, \infty}(t_{n - 1}, t_n; L^2 (\mathbb{T}))}{u} \label{t20}
\end{align}
and 
\begin{align}
  \snorm{s, 2}{\tau_{h2}^n} & \leq C \Delta t \norm{s + 1, \infty}{u^{n - 1}} \norm{W^{1, \infty}(t_{n - 1}, t_n; H^s (\mathbb{T}))}{u}. \label{t2s}
\end{align}

We can estimate $ \tau_{h3}^n $ as 
\begin{align}
  \norm{0, 2}{\tau_{h3}^n} & \leq C_{\mathrm{P}} (\norm{0,\infty}{f}, \norm{3,2}{u^{n-1}}) \Delta t \label{t30}
\end{align} 
and 
\begin{align}
  \snorm{s, 2}{\tau_{h3}^n} & \leq C_{\mathrm{P}} (\norm{s,\infty}{f}, \norm{s+3,2}{u^{n-1}}) \Delta t. \label{t3s}
\end{align} 

Finally, we estimate $ \tau_{h4\pm}^n $. Let 
\begin{align} 
  \theta_\pm (x) & = \begin{cases} 
    |x|^3 & (x_{\mathrm{min}}^\pm \leq x \leq x_{\mathrm{max}}^\pm), \\ 0 & \text{otherwise}. 
  \end{cases}
\end{align}
where $ x_{\mathrm{min}}^\pm = \min \{- \norm{0,\infty}{f} \sqrt{\Delta t} \mp \sqrt{2 \nu}, 0\} $ and $ x_{\mathrm{max}}^\pm = \max\{\norm{0,\infty}{f}\sqrt{\Delta t} \mp \sqrt{2 \nu}, 0\} $. 
Using Young's inequality and the bound $ \norm{0, 1}{\theta_\pm} \leq C $, we have 
\begin{align} 
  \norm{0, 2}{\tau_{h4\pm}^n} & \leq \frac{\Delta t}{6} \snorm{4, 2}{u^{n - 1}} \norm{0, 1}{\theta_\pm} \leq C \Delta t \snorm{4, 2}{u^{n - 1}}. \label{t40} 
\end{align} 
From a direct calculation, we also have 
\begin{align} 
  \snorm{s, 2}{\tau_{h4\pm}^n} & \leq C \Delta t \norm{s, 2}{\partial_x^4 u^{n - 1}}. \label{t4s} 
\end{align} 

By applying \eqref{t10}, \eqref{t20}, \eqref{t30} and \eqref{t40} to \eqref{t0}, we conclude that there exists a positive constant $ C_{\tau,0} $ depending on $ \norm{W^{2, \infty} (L^2)}{u} $ and $ \norm{L^{\infty} (H^4)}{u} $ such that 
\begin{align} 
  \norm{0, 2}{\tau_h^n} & \leq C_{\tau,0} \Delta t. \label{tau_0} 
\end{align} 
Similarly, by applying \eqref{t1s}, \eqref{t2s}, \eqref{t3s} and \eqref{t4s} to \eqref{t0}, we obtain a positive constant $ C_{\tau,s} $ depending on $ \norm{W^{2, \infty} (H^{s})}{u} $ and $ \norm{L^{\infty} (H^{s+4})}{u} $ such that 
\begin{align} 
  \snorm{s, 2}{\tau_h^n} & \leq C_{\tau,s} \Delta t. \label{tau_s} 
\end{align} 

\subsection{$ H^s $-stability} 

A single step of the scheme \eqref{SL-V} is carried out by applying the operators $\mathcal{S}_{\Delta t}^{\mathrm{V}}$, $\mathcal{S}_{\Delta t}^{\mathrm{A}}$, and $\mathcal{I}_{h}$ in sequence. In this subsection, we investigate the stability of the operators $\mathcal{S}_{\Delta t}^{\mathrm{V}}$ and $\mathcal{I}_{h}$, which are the linear operators among these three. The stability of $\mathcal{S}_{\Delta t}^{\mathrm{A}}$ will be addressed later in the proof of Theorem~\ref{thm:1}. 

First, we consider the stability of $ \mathcal{S}_{\Delta t}^{\mathrm{V}} $. Since $ X_{\Delta t}^1[\pm \tilde{\delta}] $ is simply a translation, we have 
\begin{align} 
  \norm{s, 2, \ast}{\mathcal{S}_{\Delta t}^{\mathrm{V}} v} \leq \frac{1}{2} \left(\norm{s, 2, \ast}{v \circ {X_{\Delta t}^1}[- \tilde{\delta}]} + \norm{s, 2, \ast}{v \circ {X_{\Delta t}^1}[+ \tilde{\delta}]}\right) = \norm{s, 2, \ast}{v} \label{b7} 
\end{align} 
and 
\begin{align} 
  \norm{s + 1, \infty}{\mathcal{S}_{\Delta t}^{\mathrm{V}} v} & \leq \frac{1}{2} \left(\norm{s + 1, \infty}{v \circ {X_{\Delta t}^1}[- \tilde{\delta}]} + \norm{s + 1, \infty}{v \circ {X_{\Delta t}^1}[+ \tilde{\delta}]}\right) \\ & = \norm{s + 1, \infty}{v}. \label{b7-2} 
\end{align} 

Next, we address the stability of the interpolation operator $ \mathcal{I}_{h} $ with respect to the $ H^s $-norm. 

\begin{lemma} \label{lem:2} 
  Suppose that the interpolation operator $ \mathcal{I}_{h} $ possesses the properties (P\ref{p:o}) and (P\ref{p:s}) for some integer $ s \leq 2 $. Furthermore, suppose that there exists a positive constant $ C_{\Delta 1} $ such that $ h^{2s} / \Delta t \leq C_{\Delta 1} $. Let $ C_I = (C_s)^2 C_{\Delta 1} $, where $ C_s $ is the constant introduced in (P\ref{p:s}). Then we have for any $ v \in H^s (\mathbb{T}) $ and $ (h, \Delta t) \in (0,1)^2 $
  \begin{align} 
    \norm{s, 2, *}{\mathcal{I}_{h} v} & \leq \left(1 + C_I \Delta t \right) \norm{s, 2, *}{v}. 
  \end{align} 
\end{lemma} 

\begin{proof}
  It is straightforward to check that, for any $ a, b \geq 0 $ and $ d > 0 $, 
  \begin{align} 
    (a + b)^2 \leq (1 + d) a^2 + (1 + 1/d) b^2. \label{2-1} 
  \end{align} 
  This yields 
  \begin{align} 
    & \norm{s, 2, *}{\mathcal{I}_{h} v}^2 \\ & = \norm{0, 2}{v + (\mathcal{I}_{h} - I) v}^2 + \snorm{s, 2}{\mathcal{I}_{h} v}^2 \\ & \leq (1 + (C_s)^2 C_{\Delta 1} \Delta t) \norm{0, 2}{v}^2 \\ 
    & \quad + \left(1 + \frac{1}{(C_s)^2 C_{\Delta 1} \Delta t}\right) \norm{0, 2}{(\mathcal{I}_{h} - I) v}^2 + \snorm{s, 2}{\mathcal{I}_{h} v}^2. \label{stability_interpolation_1} 
  \end{align}  
  For $ \phi \in H^s (\mathbb{T}) $, substituting $ v = (I - \mathcal{I}_{h}) w $ in \eqref{interpolation_error_s}, we obtain 
  \begin{align} 
    \norm{0, 2}{(I - \mathcal{I}_{h}) w} \leq C_s h^s \snorm{s, 2}{(I - \mathcal{I}_{h}) w}. \label{2-2} 
  \end{align} 
  Using this, \eqref{stability_interpolation_1} becomes 
  \begin{align}
    & \norm{s, 2, *}{\mathcal{I}_{h} v}^2 \\ & \leq (1 + (C_s)^2 C_{\Delta 1} \Delta t) \norm{0, 2}{v}^2 + \frac{1 + (C_s)^2 C_{\Delta 1} \Delta t}{C_{\Delta 1}} \frac{h^{2s}}{\Delta t} \snorm{s, 2}{(I - \mathcal{I}_{h})v}^2 + \snorm{s, 2}{\mathcal{I}_{h} v}^2. 
  \end{align} 
  Under the assumption $ h^{2s} / \Delta t \leq C_{\Delta 1} $, we have 
  \begin{align} 
    \norm{s, 2, *}{\mathcal{I}_{h} v}^2 & \leq (1 + (C_s)^2 C_{\Delta 1} \Delta t) \norm{0, 2}{v}^2 \\ 
    & \quad + (1 + (C_s)^2 C_{\Delta 1} \Delta t) \snorm{s, 2}{(\mathcal{I}_{h} - I) v}^2 + \snorm{s, 2}{\mathcal{I}_{h} v}^2. 
  \end{align} 
  By (P\ref{p:o}), we have 
  \begin{align}
    \norm{s, 2, *}{\mathcal{I}_{h} v}^2 & \leq (1 + C_I \Delta t) \left(\norm{0, 2}{v}^2 + \snorm{s, 2}{(I - \mathcal{I}_{h}) v}^2 + \snorm{s, 2}{\mathcal{I}_{h} v}^2\right) \\
    & = (1 + C_I \Delta t) \left(\norm{0, 2}{v}^2 + \snorm{s, 2}{v}^2\right) \\
    & = (1 + C_I \Delta t)  \norm{s, 2, *}{v}^2, 
  \end{align}
  which concludes the proof. 
\end{proof} 

\subsection{$ H^s $-convergence} 

Before establishing the $L^2$-error estimate, it is necessary to derive an error bound in the $H^s$-norm. Although the $L^2$-convergence rate obtained directly from the following theorem is not as sharp as the one that will be proved in the next subsection, the $H^s$ boundedness of the numerical solution is essential for the $L^2$ analysis. 

\begin{theorem} \label{thm:1} 
  Suppose that the interpolation operator $ \mathcal{I}_{h} $ satisfies (P\ref{p:o}), (P\ref{p:s}) and (P\ref{p:q}) for positive integers $ s $ and $ q $. Furthermore, assume that there exist positive constants $ C_{\Delta 2} $ and $ \varepsilon $ such that 
  \begin{align} 
    h^{2s} / \Delta t \leq C_{\Delta 2} \quad \text{and} \quad h^{2(q - s)} / \Delta t^{1 + \varepsilon} \leq C_{\Delta 2}. \label{hyp:delta} 
  \end{align} 
  Then there exist $ (h_0, \Delta t_0) \in (0, 1)^2 $ and $ C_e > 0 $ which depends on $ \norm{W^{1,\infty}(H^{s + 2})}{u} $, $ \norm{L^{\infty} (H^{s + 4})}{u} $ and $ \norm{L^{\infty}(H^{q})}{u} $ such that, if $ (h, \Delta t) \in (0, h_0) \times (0, \Delta t_0) $, the following estimate holds: 
  \begin{align} 
    \norm{s, 2, \ast}{u^n - u_h^n} \leq C_e \left(\Delta t + \frac{h^{q - s}}{\Delta t^{1/2}} + \frac{h^q}{\Delta t}\right) \label{b30} 
  \end{align} 
  for all $ n \in \{0, \ldots, N_t\} $. 
\end{theorem} 

\begin{proof} 
  In this proof, we do not explicitly indicate the dependence of the generic constant $C_{\mathrm{P}}$ on $f$. 

  Let $ e_h^n = u^n - u_h^n $. For $ n = 0 $, we have 
  \begin{align}
    \norm{s, 2, \ast}{e_h^0} = \left(\norm{0, 2}{(I - \mathcal{I}_{h}) u^0}^2 + \snorm{s, 2}{(I - \mathcal{I}_{h}) u^0}^2\right)^{1/2}. 
  \end{align}
  By (P\ref{p:s}) and (P\ref{p:q}), there exists a positive constant $ C_0 $ depending on $ \snorm{q, 2}{u} $ such that 
  \begin{align} 
    \norm{s, 2, \ast}{e_h^0} \leq C_0 h^{q - s}. 
    \label{b0}
  \end{align}
  For $ n \geq 2 $, by the second equality of \eqref{SL-V}, we have 
  \begin{align} 
    e_h^n 
    & = u^n - \mathcal{I}_{h} (\mathcal{S}_{\Delta t}^{\mathrm{A}} (\mathcal{S}_{\Delta t}^{\mathrm{V}} u_h^{n - 1})) \\
    & = \mathcal{I}_{h} (u^n - \mathcal{S}_{\Delta t}^{\mathrm{A}} (\mathcal{S}_{\Delta t}^{\mathrm{V}} u_h^{n - 1})) + (I - \mathcal{I}_{h}) u^n. 
    \label{b6}
  \end{align} 
  By \eqref{def_tau}, we can decompose the first term in the right-hand side of \eqref{b6} as 
  \begin{align}
    u^n - \mathcal{S}_{\Delta t}^{\mathrm{A}} (\mathcal{S}_{\Delta t}^{\mathrm{V}} u_h^{n - 1}) = \eta_h^n + \theta_h^n + \Delta t \, \tau_h^n, \label{b9} 
  \end{align} 
  where 
  \begin{gather}
    \begin{aligned}
      \eta_h^n & = (\mathcal{S}_{\Delta t}^{\mathrm{V}} u^{n - 1}) \circ {X_{\Delta t}^1} [f \circ \mathcal{S}_{\Delta t}^{\mathrm{A}} (\mathcal{S}_{\Delta t}^{\mathrm{V}} u_h^{n - 1})] - \mathcal{S}_{\Delta t}^{\mathrm{A}} (\mathcal{S}_{\Delta t}^{\mathrm{V}} u_h^{n - 1}), 
    \end{aligned} \\ 
    \begin{aligned} 
      \theta_h^n & = (\mathcal{S}_{\Delta t}^{\mathrm{V}} u^{n - 1}) \circ {X_{\Delta t}^1} [f \circ u^n] - (\mathcal{S}_{\Delta t}^{\mathrm{V}} u^{n - 1}) \circ {X_{\Delta t}^1} [f \circ \mathcal{S}_{\Delta t}^{\mathrm{A}} (\mathcal{S}_{\Delta t}^{\mathrm{V}} u_h^{n - 1})], \label{b9-3} 
    \end{aligned} 
  \end{gather} 
  and $ \tau_h^n $ is defined by \eqref{def_tau}. We can rewrite $ \eta_h^n $ as 
  \begin{align}
    \eta_h^n & = (\mathcal{S}_{\Delta t}^{\mathrm{V}} u^{n - 1}) \circ {X_{\Delta t}^1} [f \circ \mathcal{S}_{\Delta t}^{\mathrm{A}} (\mathcal{S}_{\Delta t}^{\mathrm{V}} u_h^{n - 1})] \\ & \quad - (\mathcal{S}_{\Delta t}^{\mathrm{V}} u_h^{n - 1}) \circ {X_{\Delta t}^1} [f \circ \mathcal{S}_{\Delta t}^{\mathrm{A}} (\mathcal{S}_{\Delta t}^{\mathrm{V}} u_h^{n - 1})] \\
    & = (\mathcal{S}_{\Delta t}^{\mathrm{V}} e_h^{n - 1}) \circ {X_{\Delta t}^1} [f \circ \mathcal{S}_{\Delta t}^{\mathrm{A}} (\mathcal{S}_{\Delta t}^{\mathrm{V}} u_h^{n - 1})]. \label{b9-2} 
  \end{align}
  Let $ \rho_h^n = (I - \mathcal{I}_{h}) u^n $. Then \eqref{b6} becomes 
  \begin{align}
    e_h^n = \mathcal{I}_{h} (\eta_h^n + \theta_h^n + \Delta t \tau_h^n) + \rho_h^n. \label{b6-1}
  \end{align}
  By \eqref{p:o}, we have 
  \begin{align}
    \snorm{s, 2}{e_h^n}^2 & = \snorm{s, 2}{\mathcal{I}_{h} (\eta_h^n + \theta_h^n + \Delta t \tau_h^n)}^2 + \snorm{s, 2}{\rho_h^n}^2. \label{b6-2} 
  \end{align}
  By an elementary inequality
  \begin{align} 
    (a_1 + a_2 + a_3)^2 \leq (1 + \Delta t) a_1^2 + \frac{4}{\Delta t} (a_2^2 + a_3^2) 
  \end{align} 
   for $ a_1, a_2, a_3 \in \mathbb{R} $ and $ \Delta t \in (0, 1) $, \eqref{b6-2} becomes 
  \begin{align} 
    \snorm{s, 2}{e_h^n}^2 & = \left(1 + \Delta t \right) \snorm{s, 2}{\mathcal{I}_{h} \eta_h^n}^2 + \frac{4}{\Delta t} \left(\snorm{s, 2}{\mathcal{I}_{h} \theta_h^n}^2 + \snorm{s, 2}{\Delta t \mathcal{I}_{h} \tau_h^n}^2\right) + \snorm{s, 2}{\rho_h^n}^2. \label{b12-1}
  \end{align} 
  By \eqref{b6-1} and the inequality 
  \begin{align}
    (a_1 + a_2 + a_3 + a_4)^2 \leq (1 + \Delta t) a_1^2 + \frac{6}{\Delta t} (a_2^2 + a_3^2 + a_4^2) 
  \end{align} 
  for $ a_1, a_2, a_3, a_4 \in \mathbb{R} $ and $ \Delta t \in (0, 1) $, we have 
  \begin{align} 
    \norm{0, 2}{e_h^n}^2 & \leq (1 + \Delta t) \norm{0, 2}{\mathcal{I}_{h} \eta_h^n}^2 \\ & \quad + \frac{6}{\Delta t} \left(\norm{0, 2}{\mathcal{I}_{h} \theta_h^n}^2 + \norm{0, 2}{\Delta t \mathcal{I}_{h} \tau_h^n}^2 + \norm{0, 2}{\mathcal{I}_{h} \rho_h^n}^2\right). \label{b12-2} 
  \end{align} 
  Combining \eqref{b12-1} and \eqref{b12-2}, we obtain 
  \begin{align} 
    \norm{s, 2, \ast}{e_h^n}^2
    & \leq (1 + \Delta t) \norm{s, 2, \ast}{\mathcal{I}_{h} \eta_h^n}^2 \\ & \quad + \frac{6}{\Delta t} \left(\norm{s, 2, \ast}{\mathcal{I}_{h} \theta_h^n}^2 + \norm{s, 2, \ast}{\Delta t \mathcal{I}_{h} \tau_h^n}^2 + \norm{0, 2}{\rho_h^n}^2\right)+ \snorm{s, 2}{\rho_h^n}^2. 
  \end{align} 
  By Lemma~\ref{lem:2}, we have 
  \begin{align}
    \norm{s, 2, \ast}{e_h^n}^2 & \leq (1 + C \Delta t) \norm{s, 2, \ast}{\eta_h^n}^2 + \frac{6}{\Delta t} \norm{s, 2, \ast}{\theta_h^n}^2 \\
    & \quad + 6 \Delta t \norm{s, 2, \ast}{\tau_h^n}^2 + \frac{6}{\Delta t} \norm{0, 2}{\rho_h^n}^2 + \snorm{s, 2}{\rho_h^n}^2. \label{b12}
  \end{align} 

  We estimate $ \norm{s, 2, \ast}{\eta_h^n}^2 $. By Lemma~\ref{lem:s2} and \eqref{b7}, we have 
  \begin{align} 
    \norm{2, s, \ast}{\mathcal{S}_{\Delta t}^{\mathrm{A}} (\mathcal{S}_{\Delta t}^{\mathrm{V}} u_h^{n - 1})} & \leq \left(1 + C_{\mathrm{P}} (\norm{s, 2, \ast}{\mathcal{S}_{\Delta t}^{\mathrm{V}} u_h^{n - 1}}) \Delta t\right) \norm{s, 2, \ast}{\mathcal{S}_{\Delta t}^{\mathrm{V}} u_h^{n - 1}} \\ 
    & \leq C_{\mathrm{P}} (\norm{s, 2, \ast}{u_h^{n - 1}}), \label{b8} 
  \end{align} 
  where we used $ \Delta t \leq 1 $. By Lemmas~\ref{lem:c} and \ref{lem:a1}, we have 
  \begin{align}
    \norm{s, 2, \ast}{\eta_h^n} 
    & = \norm{s, 2, \ast}{(\mathcal{S}_{\Delta t}^{\mathrm{V}} e_h^{n - 1}) \circ {X_{\Delta t}^1} [f \circ \mathcal{S}_{\Delta t}^{\mathrm{A}} (\mathcal{S}_{\Delta t}^{\mathrm{V}} u_h^{n - 1})]} \\
    & \leq \left(1 + C_{\mathrm{P}} (\norm{s, 2, \ast}{f \circ (\mathcal{S}_{\Delta t}^{\mathrm{A}} \mathcal{S}_{\Delta t}^{\mathrm{V}} u_h^{n - 1})})  \Delta t\right) \norm{s, 2, \ast}{\mathcal{S}_{\Delta t}^{\mathrm{V}} e_h^{n - 1}}. 
  \end{align} 
  Using \eqref{b7} and \eqref{b8}, we have 
  \begin{align} 
    \norm{s, 2, \ast}{\eta_h^n} & \leq \left(1 + C_{\mathrm{P}} (\norm{s, 2, \ast}{u_h^{n - 1}})  \Delta t\right) \norm{s, 2, \ast}{e_h^{n - 1}}. \label{b2} 
  \end{align} 

  Second, we estimate $ \norm{s, 2, \ast}{\tau_h^n} $. By \eqref{tau_0} and \eqref{tau_s}, we have 
  \begin{align} 
    \norm{s, 2, \ast}{\tau_h^n} \leq C_{\mathrm{P}} (\norm{W^{2, \infty}(H^{s})}{u}, \norm{L^{\infty}(H^{s+4})}{u}) \Delta t. \label{b28}
  \end{align} 

  Third, we estimate $ \frac{6}{\Delta t} \norm{0,2}{\rho_h^n}^2 + \snorm{s,2}{\rho_h^n}^2 $. 
  By (P\ref{p:s}) and (P\ref{p:q}), we have 
  \begin{align}
    \frac{6}{\Delta t} \norm{0, 2}{\rho_h^n}^2 + \snorm{s, 2}{\rho_h^n}^2 \leq C\left(\frac{h^{2q}}{\Delta t} + h^{2 (q - s)}\right)\snorm{q, 2}{u}^2. \label{b31}
  \end{align}

  Now we estimate $ \norm{s, 2, \ast}{\theta_h^n} $. By \eqref{b9-3} and \eqref{d4}, we have 
  \begin{align} 
    \norm{s, 2, \ast}{\theta_h^n} 
    & \leq C_{\mathrm{P}} (\norm{s + 1, \infty}{\mathcal{S}_{\Delta t}^{\mathrm{V}} u^{n - 1}}, \norm{s, 2}{u^n}, \norm{s, 2}{\mathcal{S}_{\Delta t}^{\mathrm{A}} (\mathcal{S}_{\Delta t}^{\mathrm{V}} u_h^{n - 1})}) \\ & \qquad \cdot \Delta t
    \norm{s, 2, \ast}{u^n - \mathcal{S}_{\Delta t}^{\mathrm{A}} (\mathcal{S}_{\Delta t}^{\mathrm{V}} u_h^{n - 1})}. \label{b10-4} 
  \end{align} 
  By \eqref{b8} and \eqref{b7-2}, this becomes 
  \begin{align} 
    & \norm{s, 2, \ast}{\theta_h^n} \\ & \leq C_{\mathrm{P}} (\norm{L^\infty (W^{s + 1, \infty})}{u}, \norm{s, 2}{u_h^{n - 1}}) \Delta t \norm{s, 2, \ast}{u^n - \mathcal{S}_{\Delta t}^{\mathrm{A}} (\mathcal{S}_{\Delta t}^{\mathrm{V}} u_h^{n - 1})}.  \label{b10} 
  \end{align} 
  By \eqref{b9} and \eqref{b10}, there exists a constant $ C_{\mathrm{P}1} $ which depends polynomially on $ \norm{L^\infty (W^{s + 1, \infty})}{u} $ and $ \norm{s, 2, \ast}{u_h^{n-1}} $ such that 
  \begin{align}
    \norm{s, 2, \ast}{\theta_h^n} \leq C_{\mathrm{P}1} \Delta t \left(\norm{s, 2, \ast}{\rho_h^n} + \norm{s, 2, \ast}{\theta_h^n} + \Delta t \norm{s, 2, \ast}{\tau_h^n}\right). 
    \label{b10-2}
  \end{align}
  Let $ \Delta t_2 = \min \{1/(2 C_{\mathrm{P}1}), 1\} $ and suppose $ \Delta t \leq \Delta t_2 $. 
  Then we have $ 1 / (1 - C_{\mathrm{P}1} \Delta t) \leq 1 + 2 C_{\mathrm{P}1} \Delta t $. 
  Therefore, \eqref{b10-2} becomes 
  \begin{align}
    \norm{s, 2, \ast}{\theta_h^n} \leq C_{\mathrm{P}1} \Delta t (1 + 2 C_{\mathrm{P}1} \Delta t) \left(\norm{s, 2, \ast}{\eta_h^n} + \Delta t \norm{s, 2, \ast}{\tau_h^n}\right). 
    \label{b13}
  \end{align}
  Applying \eqref{b2}, \eqref{b31} and \eqref{b13} to \eqref{b12}, we obtain 
  \begin{align} 
    \norm{s, 2, \ast}{e_h^n}^2 & \leq \left(1 + C_{\mathrm{P}} (\norm{L^{\infty} (W^{s + 1, \infty})}{u}, \norm{s, 2, \ast}{u_h^{n - 1}} ) \Delta t \right) \norm{s, 2, \ast}{e_h^{n - 1}}^2 \\ & \quad + C_{\mathrm{P}} (\norm{L^{\infty} (W^{s + 1, \infty})}{u}, \norm{s, 2, \ast}{u_h^{n - 1}}) \Delta t \norm{s, 2, \ast}{\tau_h^{n}}^2 \\ & \quad + C\left(\frac{h^{2q}}{\Delta t} + h^{2 (q - s)}\right)\snorm{q, 2}{u}^2. \label{b20} 
  \end{align} 
  By \eqref{hyp:delta}, we estimate the last term as 
  \begin{align} 
    \left(\frac{h^{2q}}{\Delta t} + h^{2 (q - s)}\right)\snorm{q, 2}{u}^2 \leq (C_{\Delta 2}^2 + C_{\Delta 2}) \Delta t^{1 + \varepsilon}\snorm{q, 2}{u}^2. \label{b27} 
  \end{align} 
  Applying \eqref{b28} and \eqref{b27} to \eqref{b20}, we obtain 
  \begin{align} 
    \norm{s, 2, \ast}{e_h^n}^2 & \leq \left(1 + C_{\mathrm{P}} (\norm{L^{\infty} (W^{s + 1, \infty})}{u}, \norm{s, 2, \ast}{u_h^{n - 1}} ) \Delta t \right) \norm{s, 2, \ast}{e_h^{n - 1}}^2 \\ 
    & \quad + C_{\mathrm{P}} (\norm{W^{2, \infty}(H^{s})}{u}, \norm{L^{\infty}(H^{s+4})}{u}, \norm{s, 2, \ast}{u_h^{n - 1}}) \Delta t^3 \\ & \quad + C \Delta t^{1 + \varepsilon}\snorm{q, 2}{u}^2. \label{b18}
  \end{align} 
  Thus, there exist positive constants $ \alpha_1 $ depending on $ \norm{L^\infty(W^{s+1, \infty})}{u} $ and $ \beta_1 $ depending on $ \norm{W^{2, \infty}(H^s)}{u} $, $ \norm{L^{\infty}(H^{s + 4})}{u} $ and $ \norm{L^{\infty}(H^q)}{u} $ such that, if $ \norm{s, 2, \ast}{u_h^{n - 1}} \leq 2 \sup_{t\in (0, T)} \norm{s, 2, \ast}{u (t)} $, then 
  \begin{align} 
    \norm{s, 2, \ast}{e_h^n}^2 & \leq \left(1 + \alpha_1 \Delta t \right) \norm{s, 2, \ast}{e_h^{n - 1}}^2 + \beta_1 \Delta t^{1 + \varepsilon}. \label{b15} 
  \end{align} 
  Let positive constants $ \Delta t_3 $ and $ h_2 $ satisfy 
  \begin{align}
    \beta_1 T \mathrm{e}^{\alpha_1 T} (\Delta t_3)^\varepsilon \leq \frac{1}{2} \sup_{t\in (0, T)} \norm{s, 2, \ast}{u (t)} \label{b16} 
  \end{align} 
  and 
  \begin{align} 
    C_0 \mathrm{e}^{\alpha_1 T} (h_0)^{2(q - s)} \leq \frac{1}{2} \sup_{t\in (0, T)} \norm{s, 2, \ast}{u (t)}, \label{b17} 
  \end{align} 
  where $ C_0 $ is the constant introduced in \eqref{b0}. 
  
  Let $ \Delta t_0 = \min\{\Delta t_1, \Delta t_2, \Delta t_3, 1\} $ and $ h_0 = \min \{h_2, 1\} $. In the following, suppose that $ (\Delta t, h) \in (0, \Delta t_0) \times (0, h_0) $. We will show 
  \begin{align}
    \norm{s, 2, \ast}{u_h^n} \leq 2 \sup_{t\in (0, T)} \norm{s, 2, \ast}{u (t)}
    \label{b14}
  \end{align}
  for $ n \in \{0, \ldots, N_t\} $ by induction. 

  By \eqref{b0} and \eqref{b17}, we can easily check \eqref{b14} for $ n = 0 $ as 
  \begin{align} 
    \norm{s, 2, \ast}{u_h^0} \leq \norm{s, 2, \ast}{u^0} + \norm{s, 2, \ast}{e_h^0} \leq \norm{s, 2, \ast}{u^0} + C_0 h^{2 (q - s)} \leq \frac{3}{2} \sup_{t\in (0, T)} \norm{s, 2, \ast}{u (t)}. 
  \end{align}
  Suppose that, for some $ k \in \{0, \ldots, N_t - 1\} $, \eqref{b14} holds for $ n = 0, \ldots, k $. 
  Then by \eqref{b15}, we have 
  \begin{align} 
    \norm{s, 2, \ast}{e_h^{n + 1}}^2 & \leq \left(1 + \alpha_1 \Delta t \right) \norm{s, 2, \ast}{e_h^n}^2 + \beta_1 \Delta t^{1 + \varepsilon} 
  \end{align} 
  for all $ n = 0, \ldots, k $. From this, it follows that 
  \begin{align}
    \norm{s, 2, \ast}{e_h^{k + 1}}^2 
    & \leq \left(1 + \alpha_1 \Delta t \right)^{k + 1} \norm{s, 2, \ast}{e_h^0}^2 
    + \sum_{n = 0}^{k} \left(1 + \alpha_1 \Delta t \right)^{n} \beta_1 \Delta t^{1 + \varepsilon}. 
    \label{b29}
  \end{align}
  Applying \eqref{b0} and the inequality 
  \begin{align}
    \left(1 + \alpha_1 \Delta t \right)^{n} \leq \left(1 + \alpha_1 \Delta t \right)^{T / \Delta t} \leq \mathrm{e}^{\alpha_1 T} \quad (0 \leq n \leq N_t)
  \end{align}
  to \eqref{b29}, we obtain 
  \begin{align} 
    \norm{s, 2, \ast}{e_h^{k + 1}}^2 
    & \leq C_0 \mathrm{e}^{\alpha_1 T} h^{2 (q - s)}
    + \beta_1 T \mathrm{e}^{\alpha_1 T} \Delta t^{\varepsilon}. 
  \end{align} 
  Using \eqref{b16} and \eqref{b17}, we have 
  \begin{align} 
    \norm{s, 2, \ast}{e_h^{k + 1}}^2 \leq \sup_{t\in (0, T)} \norm{s, 2, \ast}{u (t)}. 
  \end{align} 
  Thus, we obtain 
  \begin{align} 
    \norm{s, 2, \ast}{u_h^{k + 1}} \leq \norm{s, 2, \ast}{u^{k + 1}} + \norm{s, 2, \ast}{e_h^{k + 1}} \leq 2 \sup_{t\in (0, T)} \norm{s, 2, \ast}{u (t)}. 
  \end{align} 
  By induction, \eqref{b14} holds for all $ n = 0,\ldots, N_t $. 

  Using \eqref{b28} and \eqref{b14}, \eqref{b20} can be rewritten as 
  \begin{align} 
    \norm{s, 2, \ast}{e_h^n}^2 & \leq \left(1 + \alpha_1 \Delta t \right) \norm{s, 2, \ast}{e_h^{n - 1}}^2 + C \left(\Delta t^3 + \frac{h^{2q}}{\Delta t} + h^{2 (q - s)}\right), 
  \end{align} 
  where $ C $ depends on $ \norm{W^{2, \infty}(H^{s})}{u} $, $ \norm{L^{\infty}(H^{s+4})}{u} $ and $ \norm{L^{\infty}(H^{q})}{u} $. From this, we can deduce 
  \begin{align} 
    \norm{s, 2, \ast}{e_h^n}^2 & \leq \left(1 + \alpha_1 \Delta t \right)^{n + 1} \norm{s, 2, \ast}{e_h^0}^2 + \sum_{k = 0}^{n} C \left(\Delta t^3 + h^{2 (q - s)} + \frac{h^{2q}}{\Delta t}\right) \left(1 + \alpha_1 \Delta t \right)^k \\ 
    & \leq C_0 \mathrm{e}^{\alpha_1 T} h^{2 (q - s)} + C T \mathrm{e}^{\alpha_1 T} \left(\Delta t^2 + \frac{h^{2 (q - s)}}{\Delta t} + \frac{h^{2q}}{\Delta t^2}\right), 
  \end{align} 
  which implies \eqref{b30}. 
\end{proof} 

\subsection{Stability and convergence with respect to the weighted $ H^s $-norm} 

We recall that $ \norm{s, 2, \Delta}{\, \cdot \,} $ is defined by \eqref{nsd}. We can establish the stability of $ \mathcal{S}_{\Delta t}^{\mathrm{V}} $ with respect to $ \norm{s, 2, \Delta}{\, \cdot \,} $ as 
\begin{align} 
  \norm{s, 2, \Delta}{\mathcal{S}_{\Delta t}^{\mathrm{V}} v} & \leq \frac{1}{2} \left(\norm{s, 2, \Delta}{v \circ {X_{\Delta t}^1}[- \tilde{\delta}]} + \norm{s, 2, \Delta}{v \circ {X_{\Delta t}^1}[+ \tilde{\delta}]}\right) \\ & = \norm{s, 2, \Delta}{v}. \label{b7-3} 
\end{align} 

We can also investigate the stability of $ \mathcal{I}_{h} $ in the following lemma. 
The case for cubic spline or Hermite interpolation is already stated in Lemmas 3.6 and 4.4 in~\cite{24KaTa}. 
\begin{lemma} \label{lem:w} 
  Suppose that the interpolation operator $ \mathcal{I}_{h} $ possesses the properties (P\ref{p:o}) and (P\ref{p:s}). Then we have for any $ v \in H^s (\mathbb{T}) $ and $ (h, \Delta t) \in (0,1)^2 $, 
  \begin{align} 
    \norm{s, 2, \Delta}{\mathcal{I}_{h} v} & \leq \left(1 + (C_s)^2 \Delta t \right) \norm{s, 2, \Delta}{v}, 
  \end{align} 
  where $ C_s $ is the constant introduced in (P\ref{p:s}). 
\end{lemma} 

\begin{proof} 
  Using \eqref{2-1}, we have 
  \begin{align} 
    & \norm{s, 2, \Delta}{\mathcal{I}_{h} v}^2 \\ & = \norm{0, 2}{v + (\mathcal{I}_{h} - I) v}^2 + \frac{h^{2s}}{\Delta t} \snorm{s, 2}{\mathcal{I}_{h} v}^2 \\ 
    & \leq (1 + (C_s)^2 \Delta t) \norm{0, 2}{v}^2 + \left(1 + \frac{1}{(C_s)^2 \Delta t}\right) \norm{0, 2}{(\mathcal{I}_{h} - I) v}^2 + \frac{h^{2s}}{\Delta t} \snorm{s, 2}{\mathcal{I}_{h} v}^2. \label{w-1} 
  \end{align} 
  Applying \eqref{2-2} to \eqref{w-1}, we have 
  \begin{align} 
    & \norm{s, 2, \Delta}{\mathcal{I}_{h} v}^2 \leq (1 + (C_s)^2 \Delta t)\left(\norm{0, 2}{v}^2 + \frac{h^{2s}}{\Delta t} \snorm{s, 2}{(I - \mathcal{I}_{h})v}^2 + \frac{h^{2s}}{\Delta t} \snorm{s, 2}{\mathcal{I}_{h} v}^2\right). 
  \end{align} 
  By (P\ref{p:o}), we obtain 
  \begin{align} 
    \norm{s, 2, \Delta}{\mathcal{I}_{h} v}^2 & \leq (1 + (C_s)^2 \Delta t) \left(\norm{0, 2}{v}^2 + \frac{h^{2s}}{\Delta t} \snorm{s, 2}{v}^2\right) \\ & = (1 + (C_s)^2 \Delta t)  \norm{s, 2, \Delta}{v}^2, 
  \end{align}
  which concludes the proof. 
\end{proof} 

The following theorem is the main result of this paper. 

\begin{theorem} \label{thm:2} 
  Suppose that the interpolation operator $ \mathcal{I}_{h} $ possesses the properties (P\ref{p:o}), (P\ref{p:s}) and (P\ref{p:q}) for positive integers $ s $ and $ q $. 
  Furthermore, assume that there exist positive constants $ C_{\Delta 2} $ and $ \varepsilon $ such that $ h^{2s} / \Delta t \leq C_{\Delta 2} $ and $ h^{2(q - s)} / \Delta t^{1 + \varepsilon} \leq C_{\Delta 2} $. Then there exist $ (h_0, \Delta t_0^\prime) \in (0, 1)^2 $ and $ C_{e2} > 0 $ depending on $ \norm{W^{1,\infty}(H^{s + 2})}{u} $, $ \norm{L^{\infty} (H^{s + 4})}{u} $ and $ \norm{L^{\infty}(H^{q})}{u} $ such that, if $ (h, \Delta t) \in (0, h_0) \times (0, \Delta t_0^\prime) $, the following estimate holds: 
  \begin{align} 
    \norm{s, 2, \Delta}{u^n - u_h^n} \leq C_{e2} \left(\Delta t + \frac{h^q}{\Delta t}\right) \label{b26} 
  \end{align} 
  for all $ n \in \{0, \ldots, N_t\} $. 
  In particular, this implies the $ L^2 $-error estimate 
  \begin{align} 
    \norm{0,2}{u^n - u_h^n} \leq C_{e2} \left(\Delta t + \frac{h^q}{\Delta t}\right). \label{b26-2}
  \end{align} 
\end{theorem} 

\begin{proof} 
  For $ n = 0 $, using (P\ref{p:s}) and (P\ref{p:q}), we can estimate as 
  \begin{align} 
    \norm{s, 2, \Delta}{e_h^0} & = \left(\norm{0,2}{(I - \mathcal{I}_{h}) u_0}^2 + \frac{h^{2s}}{\Delta t} \snorm{s,2}{(I - \mathcal{I}_{h}) u_0}^2\right)^{1/2} \\ 
    & \leq C \frac{h^q}{\Delta t} \snorm{q, 2}{u}. \label{b25}
  \end{align} 
  
  Combining \eqref{b12-1} and \eqref{b12-2}, we have 
  \begin{align} 
    \norm{s, 2, \Delta}{e_h^n}^2 & = \norm{0, 2}{e_h^n}^2 + \frac{h^{2 s}}{\Delta t} \snorm{s, 2}{e_h^n}^2 \\ & \leq (1 + 3 \Delta t) \norm{s, 2, \Delta}{\mathcal{I}_{h} \eta_h^n}^2 \\ & \quad + \frac{6}{\Delta t} \left(\norm{s, 2, \Delta}{\mathcal{I}_{h} \theta_h^n}^2 + \norm{s, 2, \Delta}{\Delta t \mathcal{I}_{h} \tau_h^n}^2 + \norm{0, 2}{\rho_h^n}^2\right) + \frac{h^{2s}}{\Delta t} \snorm{s, 2}{\rho_h^n}^2. 
  \end{align}
  By Lemma~\ref{lem:w}, we have 
  \begin{align} 
    \norm{s, 2, \Delta}{e_h^n}^2 & \leq (1 + 3 \Delta t) \norm{s, 2, \Delta}{\eta_h^n}^2 + \frac{6}{\Delta t} \norm{s, 2, \Delta}{\theta_h^n}^2 \\ & \quad + 6 \Delta t \norm{s, 2, \Delta}{\tau_h^n}^2 + \frac{6}{\Delta t} \norm{0, 2}{\rho_h^n}^2 + \frac{h^{2s}}{\Delta t} \snorm{s, 2}{\rho_h^n}^2. \label{b24} 
  \end{align}

  We estimate $ \eta_h^n $. 
  By \eqref{b9-2} and \eqref{c2-2}, we have 
  \begin{align}
    \norm{s, 2, \Delta}{\eta_h^n} & = \norm{s, 2, \Delta}{(\mathcal{S}_{\Delta t}^{\mathrm{V}} e_h^n \circ {X_{\Delta t}^1} [f \circ (\mathcal{S}_{\Delta t}^{\mathrm{A}} \mathcal{S}_{\Delta t}^{\mathrm{V}} u_h^{n - 1})])} \\ & \leq \left(1 + C_{\mathrm{P}} (\norm{s, 2, \ast}{\mathcal{S}_{\Delta t}^{\mathrm{A}} \mathcal{S}_{\Delta t}^{\mathrm{V}} u_h^{n - 1}}) \Delta t\right) \norm{s, 2, \Delta}{\mathcal{S}_{\Delta t}^{\mathrm{V}} e_h^{n - 1}}. 
  \end{align}
  By \eqref{b8} and \eqref{b7-3}, we have 
  \begin{align}
    \norm{s, 2, \Delta}{\eta_h^n} \leq \left(1 + C_{\mathrm{P}} (\norm{s, 2, \ast}{u_h^{n - 1}}) \Delta t\right) \norm{s, 2, \Delta}{e_h^{n - 1}}. \label{b22-2}
  \end{align} 
  Let $ h_0 $ and $ \Delta t_0 $ be defined as in the proof of Theorem~\ref{thm:1}. In the following, suppose that $ (h, \Delta t) \in (0, h_0) \times (0, \Delta t_0) $. We have shown that $ \norm{s, 2, \ast}{u_h^n} $ is bounded as \eqref{b14}. Therefore, \eqref{b22-2} can be rewritten as 
  \begin{align} 
    \norm{s, 2, \Delta}{\eta_h^n} \leq \left(1 + C_{\mathrm{P}} (\norm{s, 2, \ast}{u^{n - 1}}) \Delta t\right) \norm{s, 2, \Delta}{e_h^{n - 1}}. \label{b22} 
  \end{align} 

  By \eqref{tau_0} and \eqref{tau_s}, we can estimate $ \norm{s, 2, \Delta}{\tau_h^n} $ as 
  \begin{align} 
    \norm{s, 2, \Delta}{\tau_h^n}^2 \leq \left((C_{\tau,0})^2 + C_{\Delta 2} (C_{\tau,s})^2\right)\Delta t^2. \label{b32} 
  \end{align} 

  By (P\ref{p:s}) and (P\ref{p:q}), we can estimate $ \frac{6}{\Delta t} \norm{0, 2}{\rho_h^n}^2 + \frac{h^{2s}}{\Delta t} \snorm{s, 2}{\rho_h^n}^2 $ as 
  \begin{align} 
    \frac{6}{\Delta t} \norm{0, 2}{\rho_h^n}^2 + \frac{h^{2s}}{\Delta t} \snorm{s, 2}{\rho_h^n}^2 \leq C \frac{h^{2q}}{\Delta t} \snorm{q,2}{u^n}. \label{b33}
  \end{align}

  We estimate $ \theta_h^n $. By \eqref{d22}, we have 
  \begin{align} 
    & \norm{s, 2, \Delta}{\theta_h^n} \\
    & = \norm{s, 2, \Delta}{(\mathcal{S}_{\Delta t}^{\mathrm{V}} u^{n - 1}) \circ {X_{\Delta t}^1}[f \circ u^n] - (\mathcal{S}_{\Delta t}^{\mathrm{V}} u^{n - 1}) \circ {X_{\Delta t}^1}[f \circ (\mathcal{S}_{\Delta t}^{\mathrm{A}} \mathcal{S}_{\Delta t}^{\mathrm{V}} u_h^{n - 1})]} \\
    & \leq C_{\mathrm{P}} (\norm{s, 2, \ast}{u^n}, \norm{s, 2, \ast}{\mathcal{S}_{\Delta t}^{\mathrm{A}} \mathcal{S}_{\Delta t}^{\mathrm{V}} u_h^{n - 1}}) \Delta t \norm{s + 1, \infty}{\mathcal{S}_{\Delta t}^{\mathrm{V}} u^{n - 1}} \\ & \qquad \cdot \norm{s, 2, \Delta}{u^n - \mathcal{S}_{\Delta t}^{\mathrm{A}} \mathcal{S}_{\Delta t}^{\mathrm{V}} u_h^{n - 1}}. 
  \end{align}
  By \eqref{b7-2}, \eqref{b9} and \eqref{b8}, this becomes 
  \begin{align} 
    \norm{s, 2, \Delta}{\theta_h^n} & \leq C_{\mathrm{P}} (\norm{L^\infty(W^{s+1, \infty})}{u}, \norm{s, 2, \ast}{u_h^{n - 1}}) \Delta t \\ & \qquad \cdot (\norm{s, 2, \Delta}{\eta_h^n} + \norm{s, 2, \Delta}{\theta_h^n} + \Delta t \norm{s, 2, \Delta}{\tau_h^n}). 
  \end{align} 
  Since we have the bound for the numerical solution as \eqref{b14}, there exists a positive constant $ C_{3} $ depending on $ \norm{L^\infty(W^{s+1, \infty}(\mathbb{T}))}{u} $ such that 
  \begin{align}
    \norm{s, 2, \Delta}{\theta_h^n} \leq C_{3} \Delta t (\norm{s, 2, \Delta}{\eta_h^n} + \norm{s, 2, \Delta}{\theta_h^n} + \Delta t \norm{s, 2, \Delta}{\tau_h^n}). 
    \label{b21}
  \end{align}
  Let $ \Delta t_4 = \min \{1 / (2 C_{3}), 1\} $ and suppose $ \Delta t \leq \Delta t_4 $. 
  Then we have $ 1 / (1 - C_{3} \Delta t) \leq 1 + 2 C_{3} \Delta t $. 
  Therefore \eqref{b21} becomes 
  \begin{align}
    \norm{s, 2, \Delta}{\theta_h^n} \leq C_{3} \Delta t (1 + 2 C_{3} \Delta t) (\norm{s, 2, \Delta}{\eta_h^n} + \Delta t \norm{s, 2, \Delta}{\tau_h^n}). \label{b23}
  \end{align} 
  Let $ \Delta t_0^\prime = \min \{\Delta t_0, \Delta t_4\} $. In the remainder of the proof, we assume that $ \Delta t \in (0, \Delta t_0^\prime) $. 

  By \eqref{b24}, \eqref{b22} and \eqref{b23}, there exists a positive constant $ \alpha_2 $ such that 
  \begin{align} 
    \norm{s, 2, \Delta}{e_h^n}^2 & \leq (1 + \alpha_2 \Delta t) \norm{s, 2, \Delta}{e_h^{n - 1}}^2 + C \Delta t \norm{s, 2, \Delta}{\tau_h^n}^2 \\ & \quad + \frac{6}{\Delta t} \norm{0, 2}{\rho_h^n}^2 + \frac{h^{2s}}{\Delta t} \snorm{s, 2}{\rho_h^n}^2. \label{b24-2} 
  \end{align} 
  Applying \eqref{b32} and \eqref{b33} to \eqref{b24-2}, we obtain 
  \begin{align}
    \norm{s, 2, \Delta}{e_h^n}^2 
    & \leq (1 + \alpha_2 \Delta t) \norm{s, 2, \Delta}{e_h^{n - 1}}^2 + \beta_2 \left(\Delta t^3 + \frac{h^{2 q}}{\Delta t}\right), 
  \end{align} 
  where $ \beta_2 $ is a positive constant depending on $ \norm{W^{2, \infty}(H^s)}{u} $, $ \norm{L^{\infty}(H^{s + 4})}{u} $ and $ \norm{L^{\infty}(H^q)}{u} $. Therefore, we have 
  \begin{align} 
    \norm{s, 2, \Delta}{e_h^n}^2 \leq (1 + \alpha_2 \Delta t)^n \norm{s, 2, \Delta}{e_h^0}^2 + \sum_{k=1}^{n} (1 + \alpha_2 \Delta t)^k \beta_2 \left(\Delta t^3 + \frac{h^{2 q}}{\Delta t}\right). 
  \end{align} 
  By \eqref{b25} and $ (1 + \alpha_2 \Delta t)^n \leq e^{\alpha_2 T} $, we have 
  \begin{align} 
    \norm{s, 2, \Delta}{e_h^n}^2 \leq e^{\alpha_2 T} \left(C \frac{h^{2q}}{\Delta t^2} \snorm{q, 2}{u_0}^2 + \beta_2 T \left(\Delta t^2 + \frac{h^{2q}}{\Delta t^2}\right)\right), 
  \end{align} 
  which implies \eqref{b26}. 

  Inequality \eqref{b26-2} follows from \eqref{b26}. 
\end{proof} 

\section{Numerical experiments} 
\label{s:4} 

We consider the problem \eqref{ad} with $ f (u) = u $, that is, the initial value problem for the Burgers equation: 
\begin{align} 
  \begin{cases} 
    \partial_t u + u \partial_x u - \nu \partial_x^2 u = 0, & \text{in} \quad \mathbb{T} \times [0, T], \\ 
    u(0) = u_0, & \text{in} \quad \mathbb{T} 
  \end{cases} 
\end{align} 
with $ T = 1 $. The initial condition is given by 
\begin{align} 
  u_0 (x) = \frac{4 \nu A \pi \sin (2 \pi x)}{1 + A \cos (2 \pi x)} 
\end{align} 
for $ A = 9/10 $ and $ \nu = 10^{-3} $. The corresponding exact solution is 
\begin{align} 
  u (x, t) = \frac{4 \nu A \pi \exp (- 4 \pi^2 \nu t) \sin (2 \pi x)}{1 + A \exp (- 4 \pi^2 \nu t) \cos (2 \pi x)}. 
\end{align} 

For the exact solution $ u $ and the numerical solution $ \{u_h^n\}_{n=0}^{N_t} $, we define the relative $ L^2 $-error by 
\begin{align} 
  \mathcal{E}_h = \frac{\norm{0,2}{u(T) - u_h^{N_t}}}{\norm{0,2}{u(T)}}, 
\end{align}
where the $ L^2 $-norms are computed using the Gauss--Legendre quadrature with seven nodes. 

We conducted numerical experiments using five different interpolation operators: linear, cubic spline, quintic spline, cubic Hermite, and quintic Hermite interpolation. Note that linear interpolation can be regarded as either a spline or Hermite interpolation of degree $ s = 1 $. However, Theorems~\ref{thm:1} and \ref{thm:2} do not cover this case. 

Three sets of experiments performed, each with a different choice of discretization parameters. In the first experiment, the spatial mesh size $h$ is varied while the time step size is fixed at $\Delta t = 10^{-3}$ (Figure~\ref{fig:1}). In this setting, we observe that the accuracy of the scheme improves as either the interpolation order increases or the mesh size decreases. 

In the second experiment, the time step size $\Delta t$ is varied while the spatial mesh size is fixed at $h = 10^{-3}$ (Figure~\ref{fig:2}). In the case of the linear interpolation, the error increases as $ \Delta t $ decreases. In contrast, when higher-order interpolation is used, the error decreases as the time step size becomes smaller. 

In the third experiment, we set $ h = 1 / 20, 1 / 40, \ldots, 1 / 320 $ and the time step size is defined as $\Delta t = (10h)^s$ (Figure~\ref{fig:3}). According to Theorem~\ref{thm:2}, when the spline or Hermite interpolation operator of $ (2 s - 1) $-th order accuracy is employed and $ h^s / \Delta t $ is kept constant, the $ L^2 $-error is expected to be $ O(h^s) $. The convergence rates observed in Figure~\ref{fig:3} are consistent with this theoretical prediction. 

\def\SlopePositionX{2.5} 
\begin{figure}
  \centering
  \begin{tikzpicture} 
    \begin{loglogaxis}[
      clip=false,  
      xlabel={Spatial mesh size $ h $}, 
      ylabel={$ L^2 $ Error}, 
      grid=both, 
      width=8cm, 
      height=8cm, 
      legend style={at={(1.05,1)}, anchor=north west}
    ]
      \addplot[
        mark=triangle,
        mark size=4, 
      ] coordinates {
        (1/20 , 8.285e-01)
        (1/40 , 6.635e-01)
        (1/80 , 4.929e-01)
        (1/160, 3.203e-01)
        (1/320, 1.574e-01)
      };
      \addlegendentry{Linear}
      \addplot[
        mark=o,
        mark size=4, 
      ] coordinates {
        (1/20 , 2.586e-02)
        (1/40 , 6.341e-03)
        (1/80 , 1.461e-03)
        (1/160, 2.892e-04)
        (1/320, 4.806e-05)
      };
      \addlegendentry{Cubic spline}
      \addplot[
        mark=square,
        mark size=4,
      ] coordinates {
        (1/20 , 1.101e-02)
        (1/40 , 2.421e-04)
        (1/80 , 2.654e-05)
        (1/160, 1.886e-05)
        (1/320, 1.852e-05)
      };
      \addlegendentry{Quintic spline}
      \addplot[
        mark=x,
        mark size=4,
      ] coordinates {
        (1/20 , 2.566e-02)
        (1/40 , 6.268e-03)
        (1/80 , 1.448e-03)
        (1/160, 2.874e-04)
        (1/320, 4.787e-05)
      }; 
      \addlegendentry{Cubic Hermite} 
      \addplot[
        mark=+,
        mark size=4,
      ] coordinates {
        (1/20 , 4.139e-03)
        (1/40 , 1.245e-04)
        (1/80 , 2.257e-05)
        (1/160, 1.863e-05)
        (1/320, 1.851e-05)
      }; 
      \addlegendentry{Quintic Hermite} 
      \draw[thick] (axis cs:\SlopePositionX*1/40,  1e-5) -- (axis cs:\SlopePositionX*1/80,  1e-5); 
      \draw[thick] (axis cs:\SlopePositionX*1/40,  1e-5) -- (axis cs:\SlopePositionX*1/40, 32e-5); 
      \draw[thick] (axis cs:\SlopePositionX*1/40,  2e-5) -- (axis cs:\SlopePositionX*1/80,  1e-5); 
      \draw[thick] (axis cs:\SlopePositionX*1/40,  8e-5) -- (axis cs:\SlopePositionX*1/80,  1e-5); 
      \draw[thick] (axis cs:\SlopePositionX*1/40, 32e-5) -- (axis cs:\SlopePositionX*1/80,  1e-5); 
      \node[anchor=west] at (axis cs:\SlopePositionX*1/40,  2e-5) {Slope $ 1 $}; 
      \node[anchor=west] at (axis cs:\SlopePositionX*1/40,  8e-5) {Slope $ 3 $}; 
      \node[anchor=west] at (axis cs:\SlopePositionX*1/40, 32e-5) {Slope $ 5 $}; 
    \end{loglogaxis} 
  \end{tikzpicture} 
  \caption{Relative $ L^2 $-error versus the spatial mesh size $ h $. The mesh is uniform in both space and time, and the time step size is fixed at $ \Delta t = 1 / 1000$. } 
  \label{fig:1} 
\end{figure}
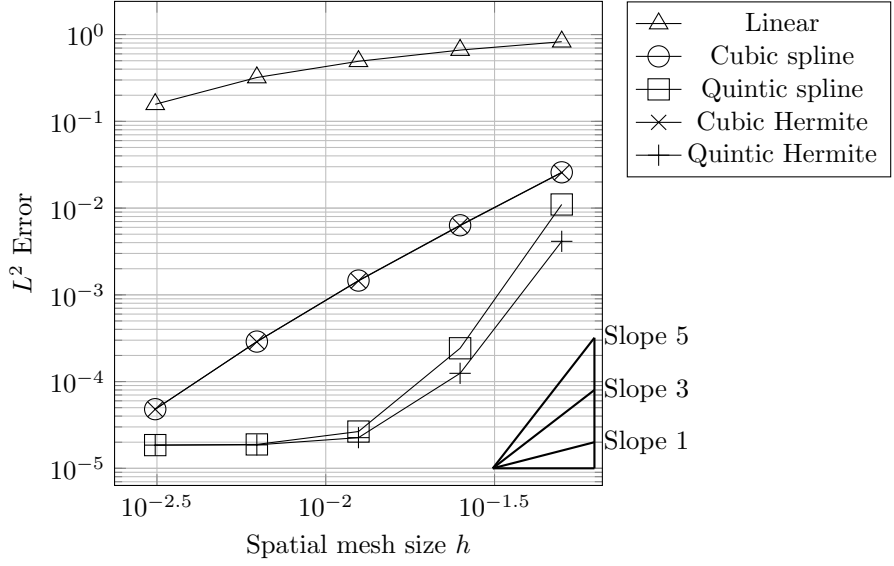 


\begin{figure}
  \centering
  \begin{tikzpicture} 
    \begin{loglogaxis}[ 
      clip=false,  
      xlabel={Spatial mesh size $ \Delta t $}, 
      ylabel={$ L^2 $ Error}, 
      grid=both, 
      width=8cm, 
      height=8cm, 
      legend style={at={(1.05,1)}, anchor=north west}
    ]
      \addplot[
        mark=triangle,
        mark size=4, 
      ] coordinates {
        (1/20 , 0.0009790)
        (1/40 , 0.0009153)
        (1/80 , 0.001246 )
        (1/160, 0.003242 )
        (1/320, 0.006767 )
      };
      \addlegendentry{Linear}
      \addplot[
        mark=o,
        mark size=4, 
      ] coordinates {
        (1/20 , 0.0009393)
        (1/40 , 0.0004660)
        (1/80 , 0.0002321)
        (1/160, 0.0001158)
        (1/320, 5.792e-05)
      };
      \addlegendentry{Cubic spline}
      \addplot[
        mark=square,
        mark size=4,
      ] coordinates {
        (1/20 , 0.0009393)
        (1/40 , 0.0004660)
        (1/80 , 0.0002321)
        (1/160, 0.000115 )
        (1/320, 5.786e-05)
      };
      \addlegendentry{Quintic spline}
      \addplot[
        mark=x,
        mark size=4,
      ] coordinates {
        (1/20 , 0.0009393)
        (1/40 , 0.0004660)
        (1/80 , 0.0002321)
        (1/160, 0.0001158)
        (1/320, 5.792e-05)
      }; 
      \addlegendentry{Cubic Hermite} 
      \addplot[
        mark=+,
        mark size=4,
      ] coordinates {
        (1/20 , 0.0009393)
        (1/40 , 0.0004660)
        (1/80 , 0.0002321)
        (1/160, 0.0001158)
        (1/320, 5.786e-05)
      }; 
      \addlegendentry{Quintic Hermite} 
      \draw[thick] (axis cs:\SlopePositionX*1/40,  1e-4) -- (axis cs:\SlopePositionX*1/80,  1e-4); 
      \draw[thick] (axis cs:\SlopePositionX*1/40,  1e-4) -- (axis cs:\SlopePositionX*1/40,  2e-4); 
      \draw[thick] (axis cs:\SlopePositionX*1/40,  2e-4) -- (axis cs:\SlopePositionX*1/80,  1e-4); 
      \node[anchor=west] at (axis cs:\SlopePositionX*1/40,  2e-4) {Slope $ 1 $}; 
    \end{loglogaxis} 
  \end{tikzpicture} 
  \caption{Relative $L^2$-errors versus the time step size $ \Delta t $. The mesh is uniform in both space and time, and the spatial mesh size is fixed at $h = 1 / 1000$. } \label{fig:2} 
\end{figure}
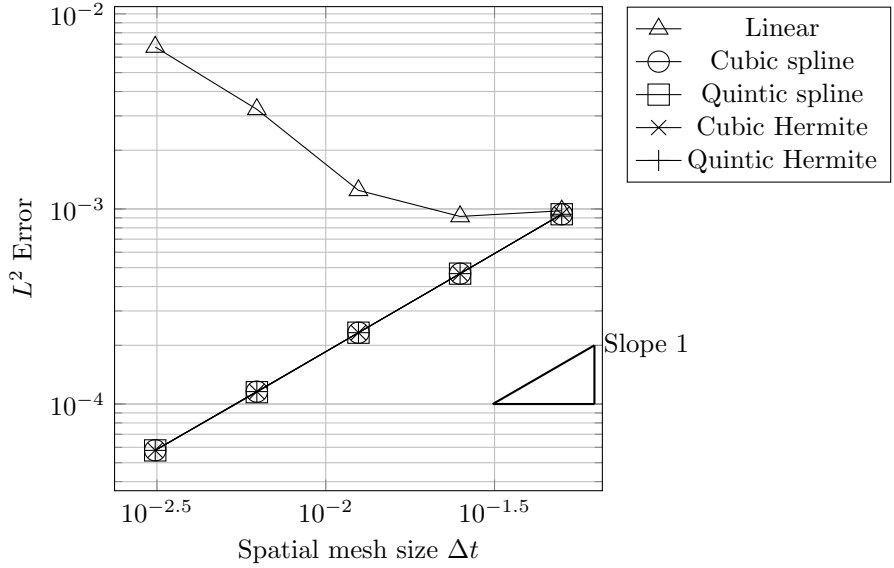 


\begin{figure} 
  \centering 
  \begin{tikzpicture} 
    \begin{loglogaxis}[
      clip=false,   
      xlabel={Spatial mesh size $ h $}, 
      ylabel={$ L^2 $ Error}, 
      grid=both, 
      width=8cm, 
      height=8cm, 
      legend style={at={(1.05,1)}, anchor=north west} 
    ]
      \addplot[
        mark=triangle,
        mark size=4, 
      ] coordinates {
        (1/20 , 0.118734   )                         
        (1/40 , 0.03963    )                         
        (1/80 , 0.0165264  )                         
        (1/160, 0.00877888 )                         
        (1/320, 0.00585609 )   
      };
      \addlegendentry{Linear}
      \addplot[
        mark=o,
        mark size=4, 
      ] coordinates {
        (1/20 , 0.0220403  )                         
        (1/40 , 0.00361572 )                         
        (1/80 , 0.000806031)                         
        (1/160, 0.000195046)                         
        (1/320, 4.83545e-05)                         
      };
      \addlegendentry{Cubic spline}
      \addplot[
        mark=square,
        mark size=4,
      ] coordinates {
        (1/20 , 0.0123807  )                         
        (1/40 , 0.000453822)                         
        (1/80 , 4.34447e-05)                         
        (1/160, 4.94654e-06)                         
        (1/320, 6.05482e-07)                  
      };
      \addlegendentry{Quintic spline}
      \addplot[
        mark=x,
        mark size=4,
      ] coordinates {
        (1/20 , 0.0122685  )                         
        (1/40 , 0.00291468 )                         
        (1/80 , 0.000758296)                         
        (1/160, 0.000191999)                         
        (1/320, 4.8163e-05 )                         
      }; 
      \addlegendentry{Cubic Hermite} 
      \addplot[
        mark=+,
        mark size=4,
      ] coordinates {
        (1/20 , 0.00330609 )                         
        (1/40 , 0.000326505)                         
        (1/80 , 3.91937e-05)                         
        (1/160, 4.75888e-06)                         
        (1/320, 5.97141e-07) 
      }; 
      \addlegendentry{Quintic Hermite} 
      \draw[thick] (axis cs:\SlopePositionX*1/40,  1e-6) -- (axis cs:\SlopePositionX*1/80,  1e-6); 
      \draw[thick] (axis cs:\SlopePositionX*1/40,  1e-6) -- (axis cs:\SlopePositionX*1/40, 8e-6); 
      \draw[thick] (axis cs:\SlopePositionX*1/40, 2e-6) -- (axis cs:\SlopePositionX*1/80,  1e-6); 
      \draw[thick] (axis cs:\SlopePositionX*1/40, 4e-6) -- (axis cs:\SlopePositionX*1/80,  1e-6); 
      \draw[thick] (axis cs:\SlopePositionX*1/40, 8e-6) -- (axis cs:\SlopePositionX*1/80,  1e-6); 
      \node[anchor=west] at (axis cs:\SlopePositionX*1/40, 2e-6) {Slope $ 1 $}; 
      \node[anchor=west] at (axis cs:\SlopePositionX*1/40, 4e-6) {Slope $ 2 $}; 
      \node[anchor=west] at (axis cs:\SlopePositionX*1/40, 8e-6) {Slope $ 3 $}; 
    \end{loglogaxis} 
  \end{tikzpicture} 
  \caption{Relative $L^2$-errors versus the spatial mesh size $ h $. The time step size is given by $ \Delta t = (10 h)^s $. The mesh is uniform in both space and time. } 
  \label{fig:3} 
\end{figure}
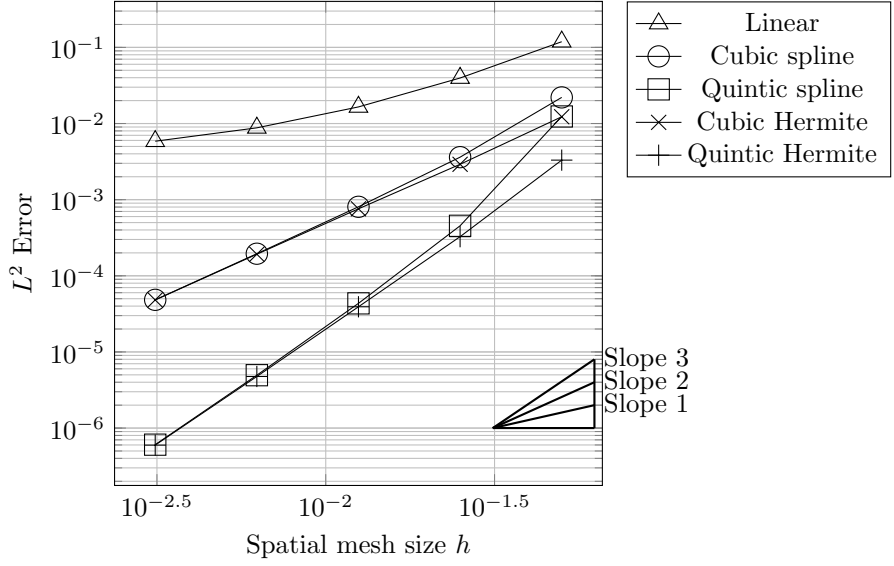 


\appendix 
\section{Auxiliary lemmas} \label{s:5} 

In this section, we prepare some lemmas for the proofs of Theorems~\ref{thm:1} and \ref{thm:2}. 

We define the Bell polynomial $ B_{l, k} $ by 
\begin{align} 
  B_{l, k} (\xi_1, \ldots, \xi_{l - k + 1}) = \sum_{\alpha \in \mathcal{A}_{s,l}} C_{l,\alpha} \prod_{j = 1}^{l - k + 1} \left(\xi_j \right)^{\alpha_j}, 
\end{align} 
where $ C_{l,\alpha} = \frac{l!}{\prod_{j=1}^{j=l}\alpha_j! (j!)^{\alpha_j}} $ and 
\begin{align} 
  \mathcal{A}_{l, k} = \left\{\alpha = (\alpha_1, \ldots, \alpha_l) \in (\mathbb{N}\cup \{0\})^l \relmiddle \sum_{i=1}^{l} i \alpha_i = l, \quad \sum_{i=1}^{l} \alpha_i = k \right\}. 
\end{align} 
In this section, we often use the Fa\`a di Bruno's formula
\begin{align}
  \partial_x^l (v \circ w) = \sum_{k = 1}^{l} \left((\partial_x^k v) \circ w\right) \mathcal{B}_{l, k} (w), \label{fdb}
\end{align} 
where $ \mathcal{B}_{l, k} (w) = B_{l, k} (\partial_x w, \ldots, \partial_x^{l - k + 1} w) $. 

\begin{lemma} \label{lem:s2} 
  Let $ v \in H^s (\mathbb{T}) $ and $ f \in W^{s, \infty} (\mathbb{T}) $ for an integer $ s \geq 2 $. Let $ \Delta t_1 $ be as defined in Lemma~\ref{lem:s}. Then there exists a positive constant $ C_{\mathrm{P}2} $ which depends polynomially on $ \norm{s, \infty}{f} $ and $ \norm{s,2,\ast}{v} $ such that if $ \Delta t < \Delta t_1 $, 
  \begin{align} 
    \norm{s,2,\ast}{\mathcal{S}_{\Delta t}^{\mathrm{A}} (v)} \leq \left(1 + C_{\mathrm{P}2} \Delta t \right) \norm{s,2,\ast} {v}. \label{a2} 
  \end{align} 
\end{lemma} 

\begin{proof} 
  In this proof, we denote $ \mathcal{S}_{\Delta t}^{\mathrm{A}} (v) $ by $ w $. We will show that, for $ k \in \{2, \ldots, s - 1\} $, 
  \begin{align} 
    \norm{k, \infty}{w} \leq \left(1 + C_{\mathrm{P}} (\norm{k, \infty}{f}, \norm{k, \infty}{v} ) \Delta t \right) \norm{k, \infty}{v}. \label{a6} 
  \end{align} 
  First, we prove \eqref{a6} for $ s = 2 $. 
  By differentiating \eqref{a1} twice and using the Fa\`a di Bruno's formula \eqref{fdb}, we have 
  \begin{align}
    \gamma \, \partial_x^2 w & = \left( (\partial_x v) \circ {X_{\Delta t}^1} [f \circ w] \right) \left\{\mathcal{B}_{2,1} ({X_{\Delta t}^1} [f \circ w]) - (f^\prime \circ w) \partial_x^2 w \Delta t \right\} \\ & \quad + \left( (\partial_x^2 v) \circ {X_{\Delta t}^1} [f \circ w] \right) \mathcal{B}_{2,2} ({X_{\Delta t}^1} [f \circ w]), 
  \end{align} 
  where $ \gamma $ is defined by \eqref{a35}. Therefore, we have 
  \begin{align}
    \norm{2, \infty}{\gamma \, \partial_x^2 w} & \leq \snorm{1, \infty}{v} \norm{0, \infty}{\mathcal{B}_{2,1} ({X_{\Delta t}^1} [f \circ w]) - (f^\prime \circ w) \partial_x^2 w \Delta t} \\ & \quad + \snorm{2, \infty}{v} \norm{0, \infty}{\mathcal{B}_{2,2} ({X_{\Delta t}^1} [f \circ w])} \label{a21}
  \end{align}
  By \eqref{a5-2}, we have 
  \begin{align} 
    \norm{0, \infty}{\mathcal{B}_{2,1} ({X_{\Delta t}^1} [f \circ w]) - (f^\prime \circ w) \partial_x^2 w \Delta t} 
    & = \Delta t \norm{0, \infty}{(f^{\prime\prime} \circ w) (\partial_x w)^2} \\
    & = \Delta t \snorm{2, \infty}{f} \snorm{1, \infty}{w}^2 \\
    & \leq \frac{9}{4} \Delta t \snorm{2, \infty}{f} \snorm{1, \infty}{v}^2. 
  \end{align}
  and 
  \begin{align}
    \norm{0, \infty}{\mathcal{B}_{2,2} ({X_{\Delta t}^1} [f \circ w])} 
    & = \norm{0, \infty}{\left(1 + (f^\prime \circ w) \partial_x w\right)^2} \\
    & \leq \left(1 + \frac{3}{2} \Delta t \snorm{1,\infty}{f} \snorm{1,\infty}{v} \right)^2, 
  \end{align}
  Applying these estimates to \eqref{a21}, we obtain 
  \begin{align}
    \norm{2, \infty}{\gamma \, \partial_x^2 w}
    & \leq C_{\mathrm{P}} (\snorm{2, \infty}{f}, \snorm{1, \infty}{v}) \Delta t \snorm{1, \infty}{v}
    + \left(1 + C_{\mathrm{P}} (\snorm{2, \infty}{f}, \snorm{1, \infty}{v}) \Delta t\right) \snorm{2, \infty}{v} \\
    & \leq \left(1 + C_{\mathrm{P}} (\norm{2, \infty}{f}, \norm{2, \infty}{v}) \Delta t\right) \norm{1, \infty}{\partial_x v}. 
    \label{w-w2-seminorm}
  \end{align}
  By \eqref{a11} and this, we obtain 
  \begin{align}
    \snorm{2, \infty}{w} 
    & \leq \norm{0, \infty}{1/\gamma} \norm{2, \infty}{\gamma \, \partial_x^2 w} \\
    & \leq \left(1 + C_{\mathrm{P}} (\norm{2, \infty}{f}, \norm{2, \infty}{v}) \Delta t\right) \norm{1, \infty}{\partial_x v}. 
  \end{align}
  In addition, we have 
  \begin{align} 
    \norm{0,\infty}{w} = \norm{0,\infty}{v \circ X_{\Delta t}^1 [f \circ w]} \leq \norm{0,\infty}{v}. \label{w-sup-norm} 
  \end{align}
  Combining \eqref{a5}, \eqref{w-w2-seminorm} and \eqref{w-sup-norm}, we conclude that \eqref{a6} holds for the case $ k = 2 $. 

  Next, suppose that for some $ l \in \{3,\ldots,s - 1\} $, \eqref{a6} holds for all $ k = 2, \ldots, l - 1 $. By the Fa\`a di Bruno's formula, we have 
  \begin{align}
    \gamma \, \partial_x^l w
    & = ((\partial_x v) \circ {X_{\Delta t}^1} [f \circ w]) \left\{\mathcal{B}_{l, 1} ({X_{\Delta t}^1} [ f \circ w]) - (f^\prime\circ w) \partial_x^l w \Delta t \right\} \\
    & \quad + \sum_{k=2}^{l} ((\partial_x^k v) \circ {X_{\Delta t}^1} [f \circ w]) \mathcal{B}_{l, k} ({X_{\Delta t}^1} [ f \circ w]). \label{a3} 
  \end{align} 

  We now estimate $ \mathcal{B}_{l, k} ({X_{\Delta t}^1} [ f \circ w]) $ ($ k = 1, \ldots, l $) by considering the following cases separately: (i) $k = 1$, (ii) $k = 2, \ldots, l-1$, and (iii) $ k = l $.
  
  First, we derive the bound for $ \norm{0, \infty}{\mathcal{B}_{l, 1} ({X_{\Delta t}^1} [ f \circ w]) - (f^\prime\circ w) \partial_x^l w} $. By the Fa\`a di Bruno's formula \eqref{fdb}, we have 
  \begin{align}
    \mathcal{B}_{l, 1} ({X_{\Delta t}^1} [ f \circ w]) - (f^\prime\circ w) \partial_x^l w \Delta t
    & = \partial_x^l (f \circ w) \Delta t - (f^\prime\circ w) \partial_x^l w \Delta t \\
    & = \sum_{k=1}^l ((\partial_x^k f) \circ w) \mathcal{B}_{l,k}(w) \Delta t - (f^\prime\circ w) \partial_x^l w \Delta t \\
    & = \sum_{k=2}^l ((\partial_x^k f) \circ w) \mathcal{B}_{l,k}(w) \Delta t, 
  \end{align}
  where we used the identity $ \mathcal{B}_{l,1}(g) = \partial_x^l g $ in the first and the last equalities. 
  Since we have the estimate 
  \begin{align}
    \norm{0,\infty}{\mathcal{B}_{l,k}(w)} \leq C_{\mathrm{P}} (\norm{l - 1, \infty}{w}) 
  \end{align}
  for $ k = 2, \ldots, l $, we obtain 
  \begin{align}
    \norm{0, \infty}{\mathcal{B}_{l, 1} ({X_{\Delta t}^1} [ f \circ w]) - (f^\prime\circ w) \partial_x^l w \Delta t}
    \leq C_{\mathrm{P}} (\norm{l, \infty}{f}, \norm{l - 1, \infty}{w}) \Delta t. 
  \end{align} 
  By the induction hypothesis, this becomes 
  \begin{align}
    \norm{0, \infty}{\mathcal{B}_{l, 1} ({X_{\Delta t}^1} [ f \circ w]) - (f^\prime\circ w) \partial_x^l w \Delta t} 
    \leq C_{\mathrm{P}} (\norm{l, \infty}{f}, \norm{l - 1, \infty}{v}) \Delta t. \label{a19} 
  \end{align} 
  
  Second, we estimate $ \mathcal{B}_{s, k} ({X_{\Delta t}^1} [ f \circ w]) $ for $ k = 2, \ldots, l - 1 $. Since every term in the Bell polynomial $ B_{l, k} (\xi_1, \ldots, \xi_{l - k + 1}) $ is of degree at least one with respect to the variables $ \xi_2, \ldots, \xi_{l - k + 1} $, we have 
  \begin{align} 
    & \norm{0, \infty}{\mathcal{B}_{l, k}({X_{\Delta t}^1} [f \circ w])} \\ 
    & = \norm{0, \infty}{B_{l, k}(1 - \partial_x (f \circ w) \Delta t, - \partial_x^2 (f \circ w) \Delta t, \ldots, - \partial_x^{l - k + 1} (f \circ w) \Delta t)} \\
    & \leq C_{\mathrm{P}} (\norm{l - k + 1, \infty}{f \circ w}) \Delta t. \label{a22} 
  \end{align} 
  By the induction hypothesis, we have $ \norm{l - k + 1, \infty}{w} \leq C_{\mathrm{P}} (\norm{l - k + 1, \infty}{f}, \norm{l - k + 1, \infty}{v}) $ for all $ k = 2, \ldots, l - 1 $. Therefore we obtain 
  \begin{align} 
    \norm{l - k + 1, \infty}{f \circ w} 
    & \leq C_{\mathrm{P}} (\norm{l - k + 1, \infty}{f}, \norm{l - k + 1, \infty}{w}) \\ 
    & \leq C_{\mathrm{P}} (\norm{l - k + 1, \infty}{f}, \norm{l - k + 1, \infty}{v}). 
  \end{align} 
  Thus, it follows from \eqref{a22} that for $ k = 2, \ldots, l - 1 $, 
  \begin{align} 
    \norm{0, \infty}{\mathcal{B}_{l, k}({X_{\Delta t}^1} [f \circ w])} & \leq C_{\mathrm{P}} (\norm{l - 1, \infty}{f}, \norm{l - 1, \infty}{v}) \Delta t. \label{a18} 
  \end{align} 

  Finally, we estimate $ \mathcal{B}_{l, l} ({X_{\Delta t}^1} [ f \circ w]) $ as 
  \begin{align} 
    \norm{0, \infty}{\mathcal{B}_{l, l} ({X_{\Delta t}^1} [ f \circ w])}& = \norm{0, \infty}{\left(\partial_x ({X_{\Delta t}^1} [ f \circ w])\right)^l} \\ & \leq \left(1 + \Delta t \snorm{1, \infty}{f} \snorm{1, \infty}{w}\right)^l. \label{a17-2} 
  \end{align}
  By \eqref{a5-2} and $ \Delta t \leq 1 $, we have 
  \begin{align}
    \norm{0, \infty}{\mathcal{B}_{l, l} ({X_{\Delta t}^1} [ f \circ w])} & \leq 1 + C_{\mathrm{P}} (\snorm{1, \infty}{f}, \snorm{1, \infty}{v}) \Delta t. \label{a17} 
  \end{align} 

  Applying these three estimates \eqref{a19}, \eqref{a18} and \eqref{a17} to \eqref{a3}, we obtain 
  \begin{align} 
    \norm{0, \infty}{\gamma \, \partial_x^l w} & \leq  \left(1 + C_{\mathrm{P}} (\norm{l, \infty}{f}, \norm{l - 1, \infty}{v}) \Delta t\right) \norm{l, \infty}{v}. 
  \end{align} 
  By \eqref{a11} and this, we have 
  \begin{align} 
    \snorm{l, \infty}{w} \leq  (1 + C_{\mathrm{P}} (\norm{l, \infty}{f}, \norm{l - 1, \infty}{v}) \Delta t) \norm{l, \infty}{v}. 
  \end{align} 
  Combining this with the induction hypothesis, we deduce
  \begin{align}
    \norm{l, \infty}{w} \leq (1 + C_{\mathrm{P}} (\norm{l, \infty}{f}, \norm{l - 1, \infty}{v}) \Delta t) \norm{l, \infty}{v}. 
  \end{align} 
  Therefore, by the induction on $ l $, the estimate \eqref{a6} holds for all $ l = 2,\ldots, s - 1 $. 

  In the following, we prove \eqref{a2}. 
  Since $ {X_{\Delta t}^1} [f \circ w] $ is an bijection from $ \mathbb{T} $ to $ \mathbb{T} $, it has the inverse, which we denote by $  ({X_{\Delta t}^1} [f \circ w])^{-1} $. For any $ \phi \in L^2(\mathbb{T}) $, we have 
  \begin{align}
    \norm{0, 2}{\phi \circ {X_{\Delta t}^1} [f \circ w]}^2 & = \int_{\mathbb{T}} \phi ({X_{\Delta t}^1} [f \circ w] (x))^2 \, \mathrm{d} x \\ 
    & = \int_{\mathbb{T}} \phi (y)^2 | ({X_{\Delta t}^1} [f \circ w])^{-1} (y) | \,\mathrm{d} y \\ 
    & \leq \norm{0, \infty}{({X_{\Delta t}^1} [f \circ w])^{-1}} \norm{0, 2}{\phi}^2. \label{a12-2}
  \end{align} 
  By \eqref{a5-2} and $ \snorm{1,\infty}{f}\snorm{1,\infty}{v}\Delta t < 1/3 $, we have 
  \begin{align}
    \norm{0, \infty}{({X_{\Delta t}^1} [f \circ w])^{-1}}^{1 / 2} \leq \left(1 - \frac{3}{2} \Delta t \snorm{1,\infty}{f}\snorm{1,\infty}{v}\right)^{-1/2} \leq 1 + 3 \Delta t \snorm{1,\infty}{f}\snorm{1,\infty}{v}. 
  \end{align}
  Using this, \eqref{a12-2} becomes 
  \begin{align}
    \norm{0, 2}{\phi \circ {X_{\Delta t}^1} [f \circ w]} \leq (1 + 3 \Delta t \snorm{1,\infty}{f}\snorm{1,\infty}{v}) \norm{0, 2}{\phi}. 
    \label{a12}
  \end{align}
  Substituting $ \phi = v $, we obtain 
  \begin{align}
    \norm{0, 2}{w} \leq (1 + 3 \Delta t \snorm{1,\infty}{f}\snorm{1,\infty}{v}) \norm{0, 2}{v}. \label{a15} 
  \end{align} 

  We estimate $ \snorm{s, 2}{w} $. By \eqref{a3}, we have 
  \begin{align} 
    & \norm{0, 2}{\gamma \, \partial_x^s w} \\ & \leq \norm{0, 2}{(\partial_x v) \circ {X_{\Delta t}^1}[f \circ w]} \norm{0, \infty}{\mathcal{B}_{s, 1} ({X_{\Delta t}^1} [f \circ w]) - (f^\prime \circ w) \partial_x^s w \Delta t} \\ 
    & \quad + \sum_{k = 2}^{s - 1} \norm{0, 2}{(\partial_x^k v) \circ {X_{\Delta t}^1}[f \circ w]} \norm{0, \infty}{\mathcal{B}_{s, k} ({X_{\Delta t}^1} [f \circ w])} \\
    & \quad + \norm{0, 2}{(\partial_x^s v) \circ {X_{\Delta t}^1}[f \circ w]} \norm{0, \infty}{\mathcal{B}_{s, s} ({X_{\Delta t}^1} [f \circ w])}. \label{a20} 
  \end{align} 

  We now estimate each term on the right-hand side of \eqref{a20}. 
  Proceeding as in \eqref{a19}, \eqref{a18} and \eqref{a17}, we can derive 
  \begin{gather} 
    \norm{0, \infty}{\mathcal{B}_{s, 1} ({X_{\Delta t}^1} [f \circ w]) - (f^\prime \circ w) \partial_x^l w \Delta t} \leq C_{\mathrm{P}} (\norm{s,\infty}{f}, \norm{s-1,\infty}{v}) \Delta t, \label{a19-2} \\
    \norm{0, \infty}{\mathcal{B}_{s, k} ({X_{\Delta t}^1} [f \circ w])} \leq C_{\mathrm{P}} (\norm{s - 1,\infty}{f}, \norm{s - 1,\infty}{v}) \Delta t \label{a18-2} 
  \end{gather} 
  for $ k = 2, \ldots, s - 1 $, and 
  \begin{align}
    \norm{0, \infty}{\mathcal{B}_{s, s} ({X_{\Delta t}^1} [f \circ w])} \leq 1 + C_{\mathrm{P}} (\snorm{1,\infty}{f}, \snorm{1,\infty}{v}) \Delta t. \label{a17-3} 
  \end{align} 
  Using \eqref{a12}, \eqref{a19-2} and \eqref{a23-emb}, we obtain the bound for the first term on the right-hand side of \eqref{a20} as 
  \begin{align} 
    & \norm{0, 2}{(\partial_x v) \circ {X_{\Delta t}^1}[f \circ w]} \norm{0, \infty}{\mathcal{B}_{s, 1} ({X_{\Delta t}^1} [f \circ w]) - (f^\prime \circ w) \partial_x^s w} \\ & \leq C_{\mathrm{P}} (\norm{s,\infty}{f}, \norm{s, 2}{v}) \Delta t \snorm{1, 2}{v}. \label{a26} 
  \end{align} 
  For the second term of \eqref{a20}, combining \eqref{a12}, \eqref{a18-2} and \eqref{a23-emb} gives, for $ k = 2,\ldots, s - 1 $, 
  \begin{align} 
    & \norm{0, 2}{(\partial_x^k v) \circ {X_{\Delta t}^1}[f \circ w]} \norm{0, \infty}{\mathcal{B}_{s, k} ({X_{\Delta t}^1} [f \circ w])} \\ & \leq C_{\mathrm{P}} (\norm{s - 1,\infty}{f}, \norm{s, 2}{v}) \Delta t \snorm{k, 2}{v}. \label{a25} 
  \end{align} 
  Regarding the third term in \eqref{a20}, by \eqref{a12}, \eqref{a17-3} and \eqref{a23-emb}, we obtain 
  \begin{align} 
    & \norm{0, 2}{(\partial_x^s v) \circ {X_{\Delta t}^1}[f \circ w]} \norm{0, \infty}{\mathcal{B}_{s, s} ({X_{\Delta t}^1} [f \circ w])} \\ & \leq \left(1 + C_{\mathrm{P}} (\snorm{1,\infty}{f}, \norm{2, 2}{v}) \Delta t\right) \snorm{s, 2}{v}. \label{a24} 
  \end{align} 
  By applying \eqref{a26}, \eqref{a25} and \eqref{a24} to \eqref{a20}, we obtain 
  \begin{align} 
    \norm{0, 2}{\gamma \, \partial_x^s w} & \leq \snorm{s, 2}{v} + C_{\mathrm{P}} (\norm{s, \infty}{f}, \norm{s, 2}{v}) \norm{s, 2}{v} \Delta t. 
  \end{align} 
  By \eqref{a11} and this, we have 
  \begin{align}
    \snorm{s, 2}{w} \leq \snorm{s, 2}{v} + C_{\mathrm{P}} (\norm{s, \infty}{f}, \norm{s, 2}{v}) \norm{s, 2}{v} \Delta t. \label{a33} 
  \end{align}
  By combining \eqref{a15} and \eqref{a33}, we obtain 
  \begin{align}
    \norm{s, 2, \ast}{w} & = \left(\norm{0, 2}{w}^2 + \snorm{s, 2}{w}^2\right)^{1/2} \\ & \leq \left[\left(\norm{0, 2}{v} + C_{\mathrm{P}} (\snorm{1, \infty}{f}, \norm{2, 2}{v}) \Delta t \norm{0, 2}{v} \right)^2 \right. \\ & \qquad \left. + \left(\snorm{s, 2}{v} + C_{\mathrm{P}} (\norm{s, \infty}{f}, \norm{s, 2}{v}) \Delta t \norm{s,2}{v}\right)^2 
    \right]^{1/2}. \label{a31} 
  \end{align} 
  We use the triangle inequality for the Euclidean norm 
  \begin{align}
    \left((a_1 + b_1)^2 + (a_2 + b_2)^2\right)^{1/2} \leq (a_1^2 + a_2^2)^{1/2} + (b_1^2 + b_2^2)^{1/2} \label{a32}
  \end{align} 
  for any $ a_1, a_2, b_1, b_2 \in \mathbb{R} $. 
  Applying \eqref{a32} to \eqref{a31}, we obtain 
  \begin{align}
    \norm{s, 2, \ast}{w} & \leq \left(\norm{0, 2}{v}^2 + \snorm{s, 2}{v}^2\right)^{1/2} + C_{\mathrm{P}} (\snorm{s, \infty}{f}, \norm{s, 2}{v}) \Delta t \norm{s,2}{v}, 
  \end{align}
  which implies \eqref{a2}. 
\end{proof} 

\begin{lemma}\label{lem:a1} 
  Let $ v \in H^s (\mathbb{T}) $ and $ f \in W^{s, \infty} (\mathbb{T}) $ for some integer $ s \geq 2 $. Then there exists a positive constant $ C_{\mathrm{P}3} $ which depends polynomially on $ \norm{s, 2, \ast}{v} $ such that 
  \begin{align} 
    \norm{s, 2, \ast}{f \circ v} \leq C_{\mathrm{P}3} \norm{s, \infty}{f}. \label{ap1} 
  \end{align} 
  
  Moreover, let $ v_1 $, $v_2 \in H^s (\mathbb{T}) $ and $ f \in W^{s + 1, \infty} (\mathbb{T}) $ for some integer $ s \geq 2 $. Then there exists a positive constant $ C_{\mathrm{P}4} $ which depends polynomially on $ \norm{s - 1, 2}{\partial_x v_1} $ and $ \norm{s - 1, 2} {\partial_x v_2} $ such that 
  \begin{align} 
    \norm{s, 2, \ast}{f \circ v_1 - f \circ v_2} \leq C_{\mathrm{P}4} \norm{s + 1, \infty}{f} \norm{s, 2, \ast}{v_1 - v_2}. \label{ap2} 
  \end{align} 
\end{lemma}

\begin{proof}
  We first estimate $ \norm{0, 2}{f \circ v} $ as 
  \begin{align}
    \norm{0, 2}{f \circ v} \leq \norm{0, \infty}{f \circ v} \leq \norm{0, \infty}{f}. \label{ap11}
  \end{align}

  To estimate the seminorm $ \snorm{s, 2}{f \circ v} $, we need the bounds for the Bell polynomials $ \mathcal{B}_{s, k} (v) $ ($ k = 1, \ldots, s $), which appears the Fa\`a di Bruno's formula \eqref{fdb}. 
  
  For $ k = 1 $, since $ \mathcal{B}_{s, 1} (v) = \partial_x^s v $, we have 
  \begin{align}
    \norm{0, 2}{\mathcal{B}_{s, 1} (v)} = \snorm{s, 2}{v}. \label{ap8}
  \end{align} 
  
  For $ k = 2, \ldots, s $, since each $ \mathcal{B}_{s, k} (v) = B_{s, k} (\partial_x v, \ldots, \partial_x^{s - k + 1} v) $ is a polynomial of $ (\partial_x v, \ldots, \partial_x^{s - 1} v) $, we have 
  \begin{align} 
    \norm{0, 2}{\mathcal{B}_{s, k} (v)} & \leq \norm{0, \infty}{\mathcal{B}_{s, k} (v)} \\ & \leq C_{\mathrm{P}} (\norm{s - 2, \infty}{\partial_x v}) \norm{s - 2, \infty}{\partial_x v} \\ & \leq C_{\mathrm{P}} (\norm{s - 1, 2}{\partial_x v}) \norm{s - 1, 2}{\partial_x v}, \label{ap10} 
  \end{align} 
  where we used the embedding \eqref{a23-emb}. 
  
  Applying \eqref{ap8} and \eqref{ap10} to the Fa\`a di Bruno's formula, we obtain 
  \begin{align} 
    \snorm{s, 2}{f \circ v} \leq \sum_{k=1}^{s} \snorm{k, \infty}{f} \norm{0, 2}{\mathcal{B}_{s, k} (v)} \leq C_{\mathrm{P}} (\norm{s, 2, \ast}{v}) \norm{s, \infty}{f}. \label{ap12} 
  \end{align} 
  Combining \eqref{ap11} and \eqref{ap12}, we obtain \eqref{ap1}. 

  Next, we consider \eqref{ap2}. We can estimate $ \norm{0, 2}{f \circ v_1 - f \circ v_2} $ as 
  \begin{align}
    \norm{0, 2}{f \circ v_1 - f \circ v_2} \leq \snorm{1, \infty}{f} \norm{0, 2}{v_1 - v_2} \label{ap3}
  \end{align} 
  (see also Lemma~4.5 in~\cite{00AcGu}). 

  Next, we estimate $ \snorm{s, 2}{f \circ v_1 - f \circ v_2} $. By the Fa\`a di Bruno's formula, we have 
  \begin{align}
    & \snorm{s, 2}{f \circ v_1 - f \circ v_2} \\
    & \leq \sum_{k=1}^s \norm{0, 2}{((\partial_x^k f) \circ v_1) \mathcal{B}_{s, k} (v_1) - ((\partial_x^k f) \circ v_2) \mathcal{B}_{s, k} (v_2)} \\
    & \leq \sum_{k=1}^s \norm{0, \infty}{(\partial_x^k f) \circ v_1} \norm{0, 2}{\mathcal{B}_{s, k} (v_1) - \mathcal{B}_{s, k} (v_2)} \\
    & \quad + \sum_{k=1}^s \norm{0, \infty}{(\partial_x^k f) \circ v_1 - (\partial_x^k f) \circ v_2} \norm{0, 2}{\mathcal{B}_{s, k} (v_2)}. \label{ap4} 
  \end{align} 

  For $ k = 1 $, since $ \mathcal{B}_{s, 1} (v) = \partial_x^s v $, we have 
  \begin{align} 
    \norm{0, 2}{\mathcal{B}_{s, 1} (v_1) - \mathcal{B}_{s, 1} (v_2)} = \snorm{s, 2}{v_1 - v_2}. \label{ap14} 
  \end{align} 
  For $ k = 2, \ldots, s $, since $ \mathcal{B}_{s, k} (v) $ is a polynomial of $ (\partial_x v, \ldots, \partial_x^{s - 1} v) $, we have 
  \begin{align} 
    \norm{0, 2}{\mathcal{B}_{s, k} (v_1) - \mathcal{B}_{s, k} (v_2)} \leq C_{\mathrm{P}} (\norm{s - 2, \infty}{\partial_x v_1}, \norm{s - 2, \infty}{\partial_x v_2}) \norm{s - 1, 2}{v_1 - v_2}. 
  \end{align} 
  By \eqref{a23-emb}, we obtain 
  \begin{align} 
    \norm{0, 2}{\mathcal{B}_{s, k} (v_1) - \mathcal{B}_{s, k} (v_2)} \leq C_{\mathrm{P}} (\norm{s - 1, 2}{\partial_x v_1}, \norm{s - 1, 2}{\partial_x v_2}) \norm{s - 1, 2}{v_1 - v_2}. \label{ap13} 
  \end{align} 
  Combining with \eqref{ap14}, we conclude that \eqref{ap13} holds for all $ k = 1, \ldots, s $. 

  Moreover, by \eqref{a23-emb}, we have for $ k = 1,\ldots,s $ 
  \begin{align}
    \norm{0, \infty}{(\partial_x^k f) \circ v_1 - (\partial_x^k f) \circ v_2} & \leq \snorm{k + 1, \infty}{f} \norm{0, \infty}{v_1 - v_2} \\
    & \leq \snorm{k + 1, \infty}{f} \norm{1, 2}{v_1 - v_2}. \label{ap6}
  \end{align} 
  By \eqref{ap8}, \eqref{ap10} and \eqref{ap6}, we have 
  \begin{align} 
      & \norm{0, \infty}{(\partial_x^k f) \circ v_1 - (\partial_x^k f) \circ v_2} \norm{0, 2}{\mathcal{B}_{s, k} (v_2)} \\ & \leq C_{\mathrm{P}} (\norm{s - 1, 2}{\partial_x v_2}) \snorm{k + 1, \infty}{f} \norm{1, 2}{v_1 - v_2}. \label{ap6-2}
  \end{align} 

  Applying \eqref{ap13} and \eqref{ap6-2} to \eqref{ap4}, we obtain 
  \begin{align}
    & \snorm{s, 2}{f \circ v_1 - f \circ v_2} \\
    & \leq \sum_{k=1}^{s} \left\{C_{\mathrm{P}} (\norm{s - 1, 2}{\partial_x v_1}, \norm{s - 1, 2}{\partial_x v_2}) \snorm{k, \infty}{f} \norm{s, 2}{v_1 - v_2} \right. \\ & \qquad \qquad + \left. C_{\mathrm{P}} (\norm{s - 1, 2}{\partial_x v_2}) \snorm{k + 1, \infty}{f} \norm{1, 2}{v_1 - v_2} \right\} \\
    & \leq C_{\mathrm{P}} (\norm{s - 1, 2}{\partial_x v_1}, \norm{s - 1, 2}{\partial_x v_2}) \norm{s + 1, \infty}{f} \norm{s, 2, \ast}{v_1 - v_2}. 
    \label{ap7}
  \end{align}
  By \eqref{ap3} and \eqref{ap7}, we obtain \eqref{ap2}. 
\end{proof} 

\begin{lemma}\label{lem:c} 
  Let $ v, w \in H^s (\mathbb{T}) $ and $ f \in W^{s,\infty} (\mathbb{T}) $ for $ s \geq 2 $. 
  Suppose that $ w $ satisfies $ \snorm{1, \infty}{w} \leq 1 / (2 \Delta t \snorm{1, \infty}{f}) $. 
  Then there exists a positive constant $ C_{\mathrm{P}5} $ which depends polynomially on $ \norm{s, \infty}{f} $ and $ \norm{s, 2}{w} $ such that 
  \begin{align} 
    \norm{s, 2, \ast}{v \circ {X_{\Delta t}^1}[f \circ w]} \leq (1 + C_{\mathrm{P}5} \Delta t) \norm{s, 2, \ast}{v}. \label{c2} 
  \end{align} 
  Furthermore, suppose there exists a positive constant $ C_{\Delta 1} $ such that $ h^{2s} / \Delta t $. Then there exists a positive constant $ C_{\mathrm{P}6} $ which depends polynomially on $ \norm{s, \infty}{f} $ and $ \norm{s, 2}{w} $ such that 
  \begin{align} 
    \norm{s, 2, \Delta}{v \circ {X_{\Delta t}^1}[f \circ w]} \leq (1 + C_{\mathrm{P}6} \Delta t) \norm{s, 2, \Delta}{v}. \label{c2-2} 
  \end{align} 
\end{lemma} 

\begin{proof} 
  In the same manner as \eqref{a12}, we can derive for any $ \phi \in L^2 (\mathbb{T}) $ 
  \begin{align} 
    \norm{0, 2}{\phi \circ {X_{\Delta t}^1} [f \circ w]} \leq (1 + \snorm{1, \infty}{f} \norm{2, 2}{w} \Delta t) \norm{0, 2}{\phi}, \label{c8} 
  \end{align} 
  provided that $  \Delta t \snorm{1, \infty}{f} \snorm{1, \infty}{w} \leq 1 / 2 $. 

  In the following, we estimate $ \snorm{s, 2}{v \circ {X_{\Delta t}^1} [f \circ w]} $. 
  By the Fa\`a di Bruno's formula \eqref{fdb}, we have 
  \begin{align} 
    \snorm{s, 2}{v \circ {X_{\Delta t}^1} [f \circ w]} & \leq \norm{0, \infty}{(\partial_x v) \circ {X_{\Delta t}^1} [f \circ w]} \norm{0, 2}{\mathcal{B}_{s, 1} ({X_{\Delta t}^1} [f \circ w])} \\ & \quad + \sum_{k=2}^{s - 1} \norm{0, 2}{(\partial_x^k v) \circ {X_{\Delta t}^1} [f \circ w]} \norm{0, \infty}{\mathcal{B}_{s, k} ({X_{\Delta t}^1} [f \circ w])} \\ & \quad + \norm{0, 2}{(\partial_x^s v) \circ {X_{\Delta t}^1} [f \circ w]} \norm{0, \infty}{\mathcal{B}_{s, s} ({X_{\Delta t}^1} [f \circ w])}. \label{c3} 
  \end{align} 

  By \eqref{ap12}, we have 
  \begin{align}
    \norm{0, 2}{\mathcal{B}_{s, 1} ({X_{\Delta t}^1} [f \circ w])} = \snorm{s, 2}{f \circ w} \Delta t \leq C_{\mathrm{P}} (\norm{s, 2, \ast}{v}) \Delta t \norm{s, \infty}{f}. \label{c13} 
  \end{align} 
  Using \eqref{a23-emb}, \eqref{c8} and \eqref{c13}, we estimate the first term in the right-hand side of \eqref{c3} as 
  \begin{align}
    & \norm{0, \infty}{(\partial_x v) \circ {X_{\Delta t}^1} [f \circ w]} \norm{0, 2}{\mathcal{B}_{s, 1} ({X_{\Delta t}^1} [f \circ w])} \\ & \leq C_{\mathrm{P}} (\norm{s, \infty}{f}, \norm{s, 2}{v}) \Delta t \norm{2, 2}{v}. \label{c6} 
  \end{align}

  in the same way as \eqref{a18} and \eqref{a17}, we can derive 
  \begin{align} 
    \norm{0, \infty}{\mathcal{B}_{s, k} ({X_{\Delta t}^1} [f \circ w])} \leq C_{\mathrm{P}} (\norm{s - 1, \infty}{f}, \norm{s - 1, \infty}{w}) \label{c12}
  \end{align} 
  for $ k = 2, \ldots, s - 1 $, and 
  \begin{align}
    \norm{0, \infty}{\mathcal{B}_{s, s} ({X_{\Delta t}^1} [f \circ w])} \leq 1 +  C_{\mathrm{P}} (\snorm{1, \infty}{f}, \snorm{1, \infty}{w}) \Delta t. \label{c11}
  \end{align} 
  Using \eqref{a23-emb}, \eqref{c8} and \eqref{c12}, we estimate the second term in the right-hand side of \eqref{c3} as 
  \begin{align} 
    & \sum_{k=2}^{s - 1} \norm{0, 2}{(\partial_x^k v) \circ {X_{\Delta t}^1} [f \circ w]} \norm{0, \infty}{\mathcal{B}_{s, k} ({X_{\Delta t}^1} [f \circ w])} \\
    & \leq C_{\mathrm{P}} (\norm{s - 1, \infty}{f}, \norm{s, 2}{w}) \Delta t \norm{s - 1, 2}{v}. \label{c5}
  \end{align}

  Using \eqref{a23-emb}, \eqref{c8} and \eqref{c11}, we estimate the third term in the right-hand side of \eqref{c3} as 
  \begin{align} 
    & \norm{0, 2}{(\partial_x^s v) \circ {X_{\Delta t}^1} [f \circ w]} \norm{0, \infty}{\mathcal{B}_{s, s} ({X_{\Delta t}^1} [f \circ w])} \\
    & \leq \left(1 + C_{\mathrm{P}} (\snorm{1, \infty}{f}, \snorm{1, \infty}{w}) \Delta t\right) \snorm{s, 2}{v}. \label{c4}
  \end{align} 

  Applying \eqref{c6}, \eqref{c5} and \eqref{c4} to \eqref{c3}, we have 
  \begin{align} 
    \snorm{s, 2}{v \circ  \circ {X_{\Delta t}^1} [f \circ w]} & \leq \snorm{s, 2}{v} + C_{\mathrm{P}} (\norm{s, \infty}{f}, \norm{s, 2}{w}) \Delta t \norm{s, 2, \ast}{v}. \label{c7} 
  \end{align}
  From \eqref{c8} and \eqref{c7}, we obtain 
  \begin{align}
    \norm{s, 2, \ast}{v \circ {X_{\Delta t}^1} [f \circ w]}^2 & \leq \left(\norm{0, 2}{v} + C_{\mathrm{P}} (\snorm{1, \infty}{f}, \norm{2, 2}{w}) \Delta t \norm{0, 2}{v}\right)^2 \\
    & \quad + \left\{\snorm{s, 2}{v} + C_{\mathrm{P}} (\norm{s, \infty}{f}, \norm{s, 2}{w}) \Delta t \norm{s, 2, \ast}{v}\right\}^2. \label{c7-2} 
  \end{align} 
  Applying \eqref{a32} to \eqref{c7-2}, we obtain 
  \begin{align}
    \norm{s, 2, \ast}{v \circ {X_{\Delta t}^1} [f \circ w]}^2 & \leq \left(\norm{s, 2, \ast}{v} +  C_{\mathrm{P}} (\norm{s, \infty}{f}, \norm{s, 2}{w}) \Delta t \norm{s, 2, \ast}{v}\right)^2, 
  \end{align} 
  which implies \eqref{c2}. 

  Combining \eqref{c8} and \eqref{c7}, we also obtain 
  \begin{align}
    & \norm{s, 2, \Delta}{v \circ {X_{\Delta t}^1}[f \circ w]}^2 \\ & \leq \left(\left((1 + \snorm{1, \infty}{f} \norm{2, 2}{w} \Delta t)\right) \norm{0, 2}{v}\right)^2 \\ & \quad + \frac{h^{2s}}{\Delta t} \left(\snorm{s, 2}{v} + C_{\mathrm{P}} (\norm{s, \infty}{f}, \norm{s, 2}{w}) \Delta t \norm{s, 2, \ast}{v}\right)^2 \\
    & \leq \left[\norm{0, 2, \Delta}{v}^2 + C_{\mathrm{P}} (\norm{s, \infty}{f}, \norm{s, 2}{w}) \Delta t \left(\norm{0, 2}{v}^2 + \frac{h^{2s}}{\Delta t} \norm{s, 2, \ast}{v}^2 \right)^{1 / 2}\right]^2. \label{c9}
  \end{align}
  Provided that $ h^{2s} / \Delta t \leq C_{\Delta 1} $, we have for any $ \phi \in H^s (\mathbb{T}) $
  \begin{align}
    \left(\norm{0, 2}{\phi}^2 + \frac{h^{2s}}{\Delta t} \norm{s, 2, \ast}{\phi}^2 \right)^{1 / 2} 
    & \leq \left((1 + C_{\Delta 1}) \norm{0, 2}{\phi}^2 + \frac{h^{2s}}{\Delta t} \snorm{s, 2}{\phi}^2 \right)^{1 / 2} \\
    & \leq (1 + C_{\Delta 1}) \norm{s, 2, \Delta}{\phi}. 
    \label{c10}
  \end{align} 
  Applying \eqref{c10} to \eqref{c9}, we obtain the desired inequality \eqref{c2-2}. 
\end{proof} 

\begin{lemma} \label{lem:d} 
  Let $ w_1 $, $ w_2 \in H^s (\mathbb{T}) $, $ f $, $ v \in W^{s + 1, \infty} (\mathbb{T}) $ for an integer $ s \geq 2 $. Then there exists a positive constant $ C_{\mathrm{P}7} $ which depends polynomially on $ \norm{s + 1, \infty}{f} $, $ \norm{s, 2, \ast}{w_1} $ and $ \norm{s, 2, \ast}{w_2} $ such that 
  \begin{align}
    & \norm{s, 2, \ast}{v \circ {X_{\Delta t}^1} [f \circ w_1] - v \circ {X_{\Delta t}^1} [f \circ w_2]} \\ & \leq C_{\mathrm{P}7} \Delta t \norm{s + 1, \infty}{v} \norm{s, 2, \ast}{w_1 - w_2}. \label{d4} 
  \end{align} 

  Furthermore, suppose that there exists a positive constant $ C_{\Delta 1} $ such that $ h^{2 s} / \Delta t \leq C_{\Delta 1} $. Then there exists a positive constant $ C_{\mathrm{P}8} $ which depends polynomially on $ \norm{s + 1, \infty}{f} $, $ \norm{s, 2, \ast}{w_1} $ and $ \norm{s, 2, \ast}{w_2} $ such that 
  \begin{align} 
    & \norm{s, 2, \Delta}{v \circ {X_{\Delta t}^1} [f \circ w_1] - v \circ {X_{\Delta t}^1} [f \circ w_2]} \\ & \leq C_{\mathrm{P}8} \Delta t \norm{s + 1, \infty}{v} \norm{s, 2, \Delta}{w_1 - w_2}. \label{d22} 
  \end{align} 
\end{lemma} 

\begin{proof} 
  By \eqref{ap2}, we have 
  \begin{align} 
    & \norm{s, 2, \ast}{v \circ {X_{\Delta t}^1} [f \circ w_1] - v \circ {X_{\Delta t}^1} [f \circ w_2]} \\ & \leq C_{\mathrm{P}} (\norm{s - 1, 2}{\partial_x {X_{\Delta t}^1} [f \circ w_1]}, \norm{s - 1, 2}{\partial_x {X_{\Delta t}^1} [f \circ w_1]}) \\ & \qquad \cdot \norm{s + 1, \infty}{v} \norm{s, 2, \ast}{{X_{\Delta t}^1} [f \circ w_1] - {X_{\Delta t}^1} [f \circ w_2]}. \label{d1}
  \end{align} 
  The right-hand side can be estimated as follows. By \eqref{ap1}, we have for $ i = 1, 2 $ 
  \begin{align} 
    \norm{s - 1, 2}{\partial_x {X_{\Delta t}^1} [f \circ w_i]} & = \norm{s - 1, 2}{1 + \partial_x (f \circ w_i) \Delta t} \\ & \leq 1 + C_{\mathrm{P}} (\norm{s, 2}{w_i}) \norm{s, \infty}{f}. \label{d2} 
  \end{align}
  Also, by \eqref{ap2}, we have 
  \begin{align}
    & \norm{s, 2, \ast}{{X_{\Delta t}^1} [f \circ w_1] - {X_{\Delta t}^1} [f \circ w_2]} \\ & = \Delta t \norm{s, 2, \ast}{f \circ w_1 - f \circ w_2} \\ 
    & \leq C_{\mathrm{P}} (\norm{s - 1, 2}{\partial_x w_1}, \norm{s - 1, 2}{\partial_x w_2}) \Delta t \norm{s + 1, \infty}{f} \norm{s, 2, \ast}{w_1 -  w_2}. \label{d3} 
  \end{align} 
  Applying \eqref{d2} and \eqref{d3} to \eqref{d1}, we obtain \eqref{d4}. 

  Next, we prove \eqref{d22}. By \eqref{ap7} and \eqref{ap3}, we have for any $ \phi_1 $, $ \phi_2 \in H^s (\mathbb{T}) $ and $ g \in W^{s + 1, \infty} (\mathbb{T}) $ 
  \begin{align} 
    \norm{s, 2, \Delta}{g \circ \phi_1 - g \circ \phi_2}^2 & = \norm{0, 2}{g \circ \phi_1 - g \circ \phi_2}^2 + \frac{h^{2s}}{\Delta t} \snorm{s, 2}{g \circ \phi_1 - g \circ \phi_2}^2 \\ & \leq C_{\mathrm{P}} (\norm{s - 1, 2}{\partial_x \phi_1}, \norm{s - 1, 2}{\partial_x \phi_2}) \norm{s + 1, \infty}{g}^2 \\ & \qquad \cdot \left( \norm{0, 2}{\phi_1 - \phi_2}^2 + \frac{h^{2s}}{\Delta t} \norm{s, 2, \ast}{\phi_1 - \phi_2}^2 \right). 
  \end{align} 
  By \eqref{c10}, we have 
  \begin{align} 
    & \norm{s, 2, \Delta}{g \circ \phi_1 - g \circ \phi_2} \\ & \leq C_{\mathrm{P}} (\norm{s - 1, 2}{\partial_x \phi_1}, \norm{s - 1, 2}{\partial_x \phi_2}) \norm{s + 1, \infty}{g} \norm{s, 2, \Delta}{\phi_1 - \phi_2}. \label{d21} 
  \end{align} 
  Substituting $ \phi_i = {X_{\Delta t}^1} [f \circ w_i] $ ($ i = 1, 2 $) and $ g = v $ into \eqref{d21}, we have 
  \begin{align} 
    & \norm{s, 2, \Delta}{v \circ {X_{\Delta t}^1} [f \circ w_1] - v \circ {X_{\Delta t}^1} [f \circ w_2]} \\ 
    & \leq C_{\mathrm{P}} (\norm{s - 1, 2}{\partial_x {X_{\Delta t}^1} [f \circ w_1]}, \norm{s - 1, 2}{\partial_x {X_{\Delta t}^1} [f \circ w_2]}) \Delta t \\ & \qquad \cdot \norm{s + 1, \infty}{v} \norm{s, 2, \Delta}{f \circ w_1 - f \circ w_2}. \label{d23} 
  \end{align} 
  Substituting $ \phi_i = w_i $ ($ i = 1, 2 $) and $ g = f $ into \eqref{d21}, we also have 
  \begin{align} 
    & \norm{s, 2, \Delta}{f \circ w_1 - f \circ w_2} \\ & \leq C_{\mathrm{P}} (\norm{s - 1, 2}{\partial_x w_1}, \norm{s - 1, 2}{\partial_x w_2}) \norm{s + 1, \infty}{f} \norm{s, 2, \Delta}{w_1 - w_2}. \label{d24} 
  \end{align} 
   Applying \eqref{d2} and \eqref{d24} to \eqref{d23}, we obtain \eqref{d22}. 
\end{proof}


\begin{thebibliography}{10}

\bibitem{00AcGu}
Y.~Achdou and J.-L. Guermond, \emph{Convergence analysis of a finite element projection/{L}agrange-{G}alerkin method for the incompressible {N}avier-{S}tokes equations}, SIAM J. Numer. Anal. \textbf{37} (2000), 799--826.

\bibitem{64AhNiWa}
J.~H. Ahlberg, E.~N. Nilson, and J.~L. Walsh, \emph{Fundamental properties of generalized splines}, Proc. Nat. Acad. Sci. U.S.A. \textbf{52} (1964), 1412--1419.

\bibitem{25Fe}
R.~Ferretti, \emph{A fully semi-{L}agrangian technique for viscous and dispersive conservation laws}, J. Comput. Phys. \textbf{526} (2025), Paper No. 113784, 16.

\bibitem{24KaTa} 
T.~Kashiwabara and H.~Takemura, \emph{Error estimates of the cubic interpolated pseudo-particle scheme for one-dimensional advection equations}, Preprint arXiv: 2402.11885v2 [math.NA].

\bibitem{02Mi}
G.~N. Milstein, \emph{The probability approach to numerical solution of nonlinear parabolic equations}, Numer. Methods Partial Differential Equations \textbf{18} (2002), 490--522.

\bibitem{00MiTr}
G.~N. Milstein and M.~V. Tretyakov, \emph{Numerical algorithms for semilinear parabolic equations with small parameter based on approximation of stochastic equations}, Math. Comp. \textbf{69} (2000), 237--267.

\bibitem{01MiTr}
\bysame, \emph{Numerical solution of the {D}irichlet problem for nonlinear parabolic equations by a probabilistic approach}, IMA J. Numer. Anal. \textbf{21} (2001), 887--917.

\bibitem{02NaTa}
M.~Nara and R.~Takaki, \emph{Stability analysis of the {CIP} scheme}, IEICE Trans. Fundamentals (Japanese Edition) \textbf{85} (2002), 950--953.

\bibitem{76Pu}
D.~K. Purnell, \emph{Solution of the advective equation by upstream interpolation with a cubic spline}, Mon. Weather Rev. \textbf{104} (1976), 42--48.

\bibitem{SV67}
M.~H. Schultz and R.~S. Varga, \emph{{$L$}-splines}, Numer. Math. \textbf{10} (1967), 345--369.

\bibitem{99SoRoBeGh}
E.~Sonnendr{\"u}cker, J.~Roche, P.~Bertrand, and A.~Ghizzo, \emph{The semi-{L}agrangian method for the numerical resolution of the {V}lasov equation}, J. Comput. Phys. \textbf{149} (1999), 201--220.

\bibitem{91StCo}
A.~Staniforth and J.~C{\^o}t{\'e}, \emph{Semi-{L}agrangian integration schemes for atmospheric models---{A} review}, Mon. Weather Rev. \textbf{119} (1991), 2206--2223.

\bibitem{85TaNiYa}
H.~Takewaki, A.~Nishiguchi, and T.~Yabe, \emph{Cubic interpolated pseudo-particle method ({CIP}) for solving hyperbolic-type equations}, J. Comput. Phys. \textbf{61} (1985), 261--268.

\bibitem{15TaFuNiIs}
D.~Tanaka, H.~Fujiwara, T.~Nishida, and Y.~Iso, \emph{Stability of the {CIP} scheme applied to linear advection equations (theory)}, Transactions of the Japan Society for Industrial and Applied Mathematics \textbf{25} (2015), 91--104, (in Japanese).

\bibitem{09ToOgYa}
K.~Toda, Y.~Ogata, and T.~Yabe, \emph{Multi-dimensional conservative semi-{L}agrangian method of characteristics {CIP} for the shallow water equations}, J. Comput. Phys. \textbf{228} (2009), 4917--4944.

\bibitem{91YaAo}
T.~Yabe and T.~Aoki, \emph{A universal solver for hyperbolic equations by cubic-polynomial interpolation. {I}. {O}ne-dimensional solver}, Comput. Phys. Comm. \textbf{66} (1991), 219--232.

\end{thebibliography}

\providecommand{\bysame}{\leavevmode\hbox to3em{\hrulefill}\thinspace}
\providecommand{\MR}{\relax\ifhmode\unskip\space\fi MR }
\providecommand{\MRhref}[2]{%
  \href{http://www.ams.org/mathscinet-getitem?mr=#1}{#2}
}
\providecommand{\href}[2]{#2}

\end{document}